\documentclass[11pt]{amsart}
\usepackage{amsmath,amsfonts,amsthm,mathrsfs,amssymb,amscd,comment,enumerate,amsxtra,url,tikz-cd,physics,tcolorbox,graphicx,stmaryrd}
\input{xypic}
\xyoption{all}
\usepackage{xcolor} 
\colorlet{mdtRed}{red!50!black}
\definecolor{dblue}{rgb}{0,0,.6}
\usepackage[colorlinks]{hyperref}
\hypersetup{linkcolor=blue,citecolor=dblue,filecolor=dullmagenta,urlcolor=mdtRed}

\newtcolorbox{mymathbox}[1][]{colback=white, sharp corners, #1}

\newtheorem{theorem}[equation]{Theorem}
\newtheorem{corollary}[equation]{Corollary}
\newtheorem{lemma}[equation]{Lemma}
\newtheorem{proposition}[equation]{Proposition}

\newtheorem*{theorem*}{Theorem}
\newtheorem*{proposition*}{Proposition}

\theoremstyle{definition}
\newtheorem{definition}[equation]{\bf Definition}

\theoremstyle{remark}
\newtheorem{remark}[equation]{\bf Remark}

\newcommand{\C}{\mathbb{C}}

\newcommand{\B}{\mathcal{B}}

\renewcommand{\P}{\mathbb{P}}
\renewcommand{\O}{\mathcal{O}}

\newcommand{\mb}[1]{\mathbb{#1}}
\newcommand{\mc}[1]{\mathcal{#1}}

\newcommand{\F}{\mathcal{F}}

\newcommand{\m}{\mathfrak{m}}

\renewcommand{\S}{\mathscr{S}}

\newcommand{\T}{\mathcal{T}}
\newcommand{\G}{\mathcal{G}}
\newcommand{\V}{\mathcal{V}}

\newcommand{\coker}{\text{coker}}
\newcommand{\im}{\text{im}\,}

\newcommand{\rk}{\text{rank}}
\newcommand{\id}{\textup{id}}
\newcommand{\Id}{\textnormal{Id}}
\newcommand{\wc}{\widetilde{C}}
\newcommand{\wt}[1]{\widetilde{#1}}

\newcommand{\tf}[0]{{\rm tf}}
\newcommand{\q}[1]{Q_{k,d}^{{\tf},#1}}
\newcommand{\ext}[0]{\textnormal{ext}}
\newcommand{\Ext}[0]{\textnormal{Ext}}
\newcommand{\EXT}[0]{\mathscr{E}xt}
\newcommand{\h}[0]{\textrm{h}}
\newcommand{\Hom}[0]{\textrm{Hom}}
\newcommand{\HOM}[0]{\mathscr{H}om}

\newcommand{\vp}[0]{\varphi}
\newcommand{\Quot}[0]{{\rm Quot}}
\newcommand{\Tau}[0]{\mathcal{T}}
\newcommand{\Tor}[0]{\textnormal{Tor}}

\newcommand{\ol}[0]{\overline}

\newcommand{\soc}[0]{\textnormal{Soc}}
\newcommand{\len}[0]{\textnormal{length}}
\numberwithin{equation}{section}

\renewcommand \subsection[1]{
	\refstepcounter{equation}
	\refstepcounter{subsection}
	\noindent {\bf \arabic{section}.\arabic{subsection}.}{\bf #1}.
}

\begin{document}

\title[Irreducibility of some Quot schemes on Nodal curves]{Irreducibility of some Quot schemes on Nodal curves}
\author[Parvez Rasul]{Parvez Rasul}

\address{School of Mathematics, Korea Institute for Advanced Study, 85 Hoegiro, Dongdaemun-gu, Seoul 02455, Republic of Korea}

\email{\href{mailto:rasulparvez@kias.re.kr}{rasulparvez@kias.re.kr}}



\begin{abstract}
Let $C$ be an integral projective nodal curve over $\C$, 
of arithmetic genus $g \geqslant 2$.
Let $E$ be a vector bundle on $C$ of rank $r$ and degree $e$.
Let $\Quot_{C/\C}(E,k,d)$ denote the Quot scheme of quotients of $E$ of rank $k$ and degree $d$. 
We show that $\Quot_{C/\C}(E,k,d)$ is irreducible for $d \gg 0$.
\end{abstract}

\maketitle

\section{Introduction}
Let $C$ be a reduced and irreducible projective curve over $\C$ of arithmetic genus $g$.
For a coherent sheaf $F$ on $C$ of rank $k$, the degree of $F$ is defined as
\begin{equation}\label{degree definition}
    \deg F := \chi(F) - k(1-g)\,.
\end{equation}
Let $E$ be a vector bundle on $C$ of rank $r$ and degree $e$. 
Let $\Quot_{C/\C}(E,k,d)$ denote Grothendieck's Quot scheme of coherent quotients of $E$ of rank $k$ and degree $d$.
We will use the simpler notation $Q_{k,d}(E)$ to denote the Quot scheme $\Quot_{C/\C}(E,k,d)$.
Quot schemes have played a central role in the study of vector bundles on curves, in particular for constructing and understanding the geometry of the moduli spaces.
Quot schemes also arise as a compactification of the space of maps into Grassmannians.
Thus, Quot schemes also arise naturally in enumerative geometry. 
One of the main results in this direction, due to \cite{Ber94}, shows that when the Quot scheme is irreducible, intersection pairings on the Quot scheme compute the number of certain maps from $C$ to the Grassmannian.
Intersection theory on Quot schemes has been further developed and applied in various contexts, for example, in \cite{MO10}, \cite{MO07a}, \cite{MO07b}, \cite{MO07c}, and \cite{Mar25}.
Beyond curves, Quot schemes on higher-dimensional varieties play an important role in Donaldson–Thomas theory and in the study of enumerative invariants of surfaces and threefolds. Some recent works in these directions include \cite{OP21}, \cite{AJLOP}, \cite{Boj25}, \cite{Ric20}, \cite{SR21}, and \cite{FMR21}.

When $C$ is a smooth curve, the Quot scheme 
$Q_{0,d}(E)$ of torsion quotients of $E$ is a smooth projective variety, which has been studied extensively.
We refer the reader to the recent works  \cite{BDH15},\cite{BGS22}, \cite{OS23}, \cite{OP21} and references therein.
There are a few results when $C$ is singular.
In \cite{Reg82}, the author shows that if $C$ is embeddable in a smooth surface then 
the Quot scheme $Q_{0,d}(\O^{\oplus r}_{C})$  is irreducible. It is also proved that the compactification of moduli spaces of vector bundles on such curves is irreducible.
Some sort of converse is proved in \cite{KK81}, which shows that if $C$ is not embeddable in a smooth surface then the Quot scheme 
$Q_{0,d}(\omega)$ is reducible, where $\omega$ is the dualizing sheaf on $C$.
The authors use this to show that the compactified Jacobian of such $C$ is reducible.
In \cite{Ran05}, the punctual Hilbert scheme parametrizing degree-$d$ subschemes of $C$ supported at the node is shown to be a chain of $d-1$ rational curves meeting transversely.
It is also proved that the relative Hilbert scheme $\textrm{Hilb}(X/B)$ parametrizing length-$m$ subschemes in the fibres of the family $X/B$ is smooth, where $X$ is a 1-parameter smooth family of curves over $B$.

We come to the case $k>0$ and when the curve $C$ is smooth.
In the genus 0 case, when $E$ is the trivial bundle, Str\o mme proved that the Quot scheme $Q_{k,d}(\O^{\oplus r}_{\P^1})$ is smooth and irreducible and computed its Picard group in \cite{Str}.
Let $C$ be smooth of genus $g > 0$. When $E$ is trivial, it is proved in \cite{BDW} that the Quot scheme $Q_{k,d}(\O_C^{\oplus r})$ is irreducible and generically smooth for $d \gg 0$. 
The authors also show that $Q_{k,d}(\O_C^{\oplus r})$ is a local complete intersection for large $d$.
For any vector bundle $E$ on $C$, it is proved in \cite{PR03} that the Quot scheme $Q_{k,d}(E)$ is irreducible and generically smooth when $d \gg 0$.
The Picard group of the Quot scheme $Q_{k,d}(E)$ is computed in \cite{gs24} for $d \gg 0$.
The authors also show that if $d \gg 0$ then the Quot scheme is an integral, normal variety which is a local complete intersection and locally factorial.

It is natural to ask what happens in the case of semistable curves (limit objects in the moduli space of curves).
In this article we extend some of the above results to the case when $C$ is an irreducible nodal curve.
In \cite{Gol19}, the author defines a nice intersection theory on the Quot scheme over nodal curves by relating it to the Quot schemes on the normalization.
In order to do this, the author shows that, for a very general vector bundle $E$ on $C$, the unique top dimensional component of the Quot scheme $Q_{k,d}(E)$ is generically smooth and contains all the points corresponding to torsion-free quotients.
In this article, we prove the following theorem.
\begin{theorem}[Theorem \ref{theorem main}]
    Let $g$ and $n$ be integers such that $g>n>0$.
    Let $C$ be a reduced irreducible projective curve over $\C$ of arithmetic genus $g$ with nodes at $n$ points and smooth elsewhere.
    Let $E$ be a vector bundle on $C$ of rank $r$ and degree $e$.
    Let $k$ be an integer such that $0<k<r$.
    Then there is a number $d_Q(E,k)$ such that if $d \geqslant d_Q(E,k)$ then the Quot scheme $Q_{k,d}(E)$ is irreducible and generically smooth of dimension $dr-ke+k(r-k)(1-g)$.
\end{theorem}

We will discuss our strategy briefly.
The strategy we use is similar to that in \cite[Section 6]{PR03},
however, some modifications and new inputs are required.
It will be convenient for the reader if we first discuss the strategy of \cite{PR03}, where the theorem is proved for smooth curves.
\begin{enumerate}[(1)]
\item\label{step upper bound} In \cite[Theorem 4.1]{PR03}, the authors compute an upper bound for the dimension of the Quot scheme $Q_{k,d}(E)$ which is as follows:
$$ \dim Q_{k,d}(E) \leqslant k(r-k)+(d-m(E,k))r, \quad \textrm{ for all } d\geqslant m(E,k)\,.$$
The number $m(E,k)$ is the minimal degree of a quotient of $E$ of rank $k$.

\item\label{step stable smooth} The authors prove that the Quot scheme always contains a stable locally free quotient if $d$ is large enough.
Then using the irreducibility of the moduli space of stable locally free sheaves on $C$, it is concluded that
there is a unique component of the Quot scheme whose general point corresponds to a stable bundle quotient.
Let us call this component $\mc P$.
The dimension of $\mc P$ is $dr-ke+k(r-k)(1-g)$.

\item\label{step locally free smooth} The authors prove that for large $d$, a general vector bundle quotient of $E$ can be deformed to a stable vector bundle quotient.
A key ingredient for this step is the upper bound for the dimension of the Quot scheme given in step \eqref{step upper bound}.
This proves that the component $\mc P$ contains all the points corresponding to locally free quotients.

\item\label{step torsion locus smooth} Let $Z$ denote the locus of points in the Quot scheme which correspond to quotients with torsion.
The authors show that when $d$ is large, 
$\dim Z < dr-ke+k(r-k)(1-g)$.
There is a cohomological lower bound for the dimension of any component of the Quot scheme, which is $dr-ke+k(r-k)(1-g)$ in this case.
This proves that the points in $Z$ cannot be general in any component of the Quot scheme.
This shows that a general point in any component corresponds to a locally free quotient and so the component has to be $\mc P$.
Thus, the Quot scheme $Q_{k,d}(E)$ is irreducible when $d$ is sufficiently large.
\end{enumerate}
In the case of nodal curves, we proceed as follows.
The first two steps are very similar to that in \cite{PR03}.
However, dealing with quotients with torsion is substantially harder in the nodal case.
\begin{enumerate}[(a)]
\item\label{step stable locus} In a similar fashion, as in \cite{PR03}, we show that there is a unique component of the Quot scheme whose generic point corresponds to a stable locally free quotient, when $d$ is large.
Let us call this component $\mc P$.
This component has dimension $dr-ke+k(r-k)(1-g)$.

\item\label{step torsion-free locus} Using \cite[Lemma 2.4]{bh06},
we show that for large $d$, a general locally free quotient in the Quot scheme can be deformed to a stable locally free quotient.
This proves that the component $\mc P$ contains all the points corresponding to locally free quotients, when $d$ is large.
Similarly using \cite[Lemma 2.3]{bh06},
we show that for large $d$, a general torsion-free quotient in the Quot scheme can be deformed to a locally free quotient.
This proves that $\mc P$ contains all the points corresponding to torsion-free quotients, when $d$ is large.
To show the existence of deformations mentioned above, some $H^1$ vanishing results are required.
These results rely on dimension estimates of a few subsets of the Quot schemes. 
For example, the codimension of the locus $Y_{\mu_0} = \{[E \to F] \in Q_{k,d}^{\tf}(E) : \mu_{\min}(F) <  \mu_0\}$ increases linearly with $d$ when $d \gg 0$ (Lemma \ref{lemma dimension Ymu}).
This is implicit in \cite[Lemma 6.3]{PR03}, and is also used in the proof of \cite[Lemma 4.9]{gs24}.

\item\label{step torsion locus} 
It remains to deal with the locus of quotients with torsion.
Let $Z$ denote the locus of points in the Quot scheme which correspond to quotients with torsion. 
It is natural to try and follow step \eqref{step torsion locus smooth} to deal with this case. 
For this we need a good cohomological lower bound.
In the case of smooth curves, any coherent sheaf is direct sum of a locally free sheaf and a torsion sheaf.
However, for nodal curves, this is not the case.
As a result, computing a good cohomological lower bound becomes difficult.
Using \cite[Lemma 2]{EL99}, we compute the following lower bound 
\begin{equation*}\label{expected dimension nodal}
    \textrm{\textit{cohomological lower bound}} =  dr-ke+k(r-k)(1-g) - (r-k)(2r+k)\,.
\end{equation*}
Unfortunately, unlike in the case of smooth curves (step \eqref{step torsion locus smooth}), the upper bound on the dimension of $Z$ is not less than the \textit{cohomological lower bound}. 
Thus, the strategy in step \eqref{step torsion locus smooth} fails to generalize to the nodal case.

Instead, we do the following.
Let $Z^{\rm g}$ denote the open subset of $Z$ which contains all those quotients $[E \to F]$ such that $F$ is free at the nodes.
We observe that the dimension of $Z^{\rm g}$ is less than the \textit{cohomological lower bound} and using this we conclude that $Z^{\rm g}$ is inside the component $\mc P$.
Define $Z^{\rm b} := Z \setminus Z^{\rm g}$.
We show that if $Z^{\rm b}$ is open in a component say $\mc W (\neq \mc P)$ of the Quot scheme, 
then a general point of $\mc W$ is in the closure of $Z^{\rm g}$.
This is a contradiction as $Z^{\rm g} \subset \mc P$.
This proves that $Z^{\rm b}$ actually lies in the component $\mc P$. 
To show that a general point $[q:E \to F]$ of $\mc W$ is in the closure of $Z^{\rm g}$, 
ideally, we would want a deformation of quotients whose general member corresponds to a point in $Z^{\rm g}$ 
and a special member corresponds to the point $[q: E \to F]$.
However, as $F$ has torsion, 
such a deformation cannot be constructed using 
the same technique as in step \eqref{step torsion-free locus}.
Rather, we deform the inclusion $\ker q \hookrightarrow E$ to an inclusion which gives a point in $Z^{\rm g}$.
To construct such a deformation, we will require $\Ext^1(\ker q, E)$ to vanish for a general point  $[q: E \to F]$ in $\mc W$.
In Proposition \ref{proposition Ext1(S,E) vanishinig}, we prove that this condition is indeed satisfied if $Z^{\rm b}$ is open in $\mc W$.
This proof uses dimension estimate of the subset
$X_{\mu_0} = \{ [q:E \to F] \in Q_{k,d}(E): F \text{ is torsion free and } \mu_{\max}(\ker q) > \mu_0 \}$ (Lemma \ref{lemma dimension Xmu}), 
which is obtained by applying Lemma \ref{lemma dimension Ymu} to the Quot scheme $Q_{r-k,d-e}(E^\vee)$.
\end{enumerate}

\vspace{.2cm}

It would be interesting to know if a similar result holds when $C$ is an arbitrary integral projective curve. 
From \cite{Vak06} it is known that, in general the Hilbert schemes can be quite pathological.
In \cite{gs24}, it is proved that 
for large $d$, the Quot scheme $Q_{k,d}(E)$ over a smooth curve is an integral variety which is a local complete intersection and locally factorial.
However, a similar argument does not extend to the case of Quot schemes over nodal curves due to the non-vanishing of certain Ext groups. We briefly discuss this issue in the final section of the article.
It would be interesting to know what kinds of singularities occur in the Quot schemes over nodal curves.
In \cite{MO07c}, the authors construct a perfect obstruction theory for Quot schemes over a smooth curve, which induces a virtual fundamental class of the expected dimension.
However, again, the perfect obstruction theory fails to extend to the case of Quot schemes over nodal curves, due to a similar problem.
It remains unknown whether a virtual fundamental class exists in this setting.

\vspace{.4cm}
The organization of the paper is as follows.
In section \ref{section prelim}, we go through some preliminary lemmas regarding the sheaves on nodal curves
and then we prove the main lemma which estimates a $(\hom-\ext^1)$ term.
This lemma directly gives us the cohomological lower bound for the Quot scheme which is mentioned in step $\eqref{step torsion locus}$.
In section \ref{section dimension of subsets}, we compute the dimension estimates of a few subsets of the Quot scheme, which will help to prove some $H^1$ and $\Ext^1$ vanishing results later.
Section \ref{section component stable locally free} is occupied by the step \eqref{step stable locus} mentioned above.
This section is self-contained and does not involve any results from the first two sections.
In section \ref{section component of torsion-free quotients}, we carry out the step \eqref{step torsion-free locus}.
In section \ref{section Irreducibility of the Quot scheme}, we deal with the locus of points corresponding to the quotients with torsion and carry out step \eqref{step torsion locus} to prove the main theorem (Theorem \ref{theorem main}).

\phantom{.}

\noindent
\textbf{Acknowledgements.} I thank Usha Bhosle, Indranil Biswas and Chandranandan Gangopadhyay for useful discussions. 
I thank my supervisor, Ronnie Sebastian, for suggesting this problem and for several useful discussions. 
I also thank the referee for comments and suggestions that greatly improved the manuscript.
This work was supported by the Prime Minister’s
Research Fellowship (PMRF ID 1301167) funded by the Ministry of Education, Government
of India.

\section{Preliminaries}\label{section prelim}
Let $C$ be an integral, projective curve over $\C$ of arithmetic genus $g \geqslant 2$ with exactly one singular point $x \in C$.
Further, we assume that the singularity at $x$ is an ordinary node.
For any point $c \in C$, $\O_{C,c}$ will denote the local ring of $C$ at $c$ and $\m_c$ will denote the maximal ideal of $\O_{C,c}$. 
The residue field of $\O_{C,c}$ will be denoted by $\kappa(c)$.
We will also use $\kappa(c)$ to mean the skyscraper sheaf on $C$ supported at the point $c$.

\begin{definition}[Type of a sheaf]\label{torsion-free structure}
Let $F$ be a coherent sheaf on $C$ of rank $k$ such that the stalk $F_x$ at the node $x$ is torsion-free. 
From \cite[Part 8, Proposition 2, page 164]{Ses82}, 
we know that there exists an integer $a$ satisfying $0\leqslant a \leqslant k$ 
such that the stalk $F_x \cong \O_{C,x}^{\oplus k-a} \oplus \m_x^{\oplus a}$. 
We call such $F$ torsion-free at node of type $a$.
If $F$ is torsion free at node of type $a$ then $\dim_\C(F \otimes \kappa(x)) = k+a$.
\end{definition}

\begin{lemma}\label{Lemma hom(B,A) - ext1(B,A) torsionfree}
    Let $A$ and $B$ be sheaves on $C$ of rank $r_A, r_B$ and degree $d_A, d_B$ respectively.
    Assume that $A$ and $B$ are 
    torsion-free of type $a$ and $b$ respectively. Then
    $$ \hom(B,A) - \ext^1(B,A) 
    = r_Bd_A - r_Ad_B + r_Ar_B(1-g) -ab\,.$$
\end{lemma}
\begin{proof}
As $A$ and $B$ both are torsion-free, using \cite[Lemma 2.5(B)]{bh06} we get that
$$\hom(B,A) - \ext^1(B,A) = \chi(\HOM(B,A)) -2ab\,.$$
And \cite[Lemma 2.5(B)]{bh06} says that the degree of the sheaf $\HOM(B,A)$ is 
$$r_Bd_A - r_Ad_B + ab\,.$$
So using \eqref{degree definition}, we have
\begin{equation*}
\begin{split}
    \hom(B,A) - \ext^1(B,A) 
    &= \chi(\HOM(B,A)) -2ab \\
    &= r_Bd_A - r_Ad_B + ab+r_Ar_B(1-g) -2ab\\
    &= r_Bd_A - r_Ad_B + r_Ar_B(1-g)-ab\,.
    \end{split}
\end{equation*}
\end{proof}

\begin{lemma}\label{lemma type of kernel}
Let 
$$ 0 \to B \to E \to A \to 0$$
be a short exact sequence of torsion-free sheaves and $E$ be locally free.
Then the type of $A$ and the type of $B$ are equal.
\end{lemma}
\begin{proof}
Let the type of $A$ be $a$ and the type of $B$ be $b$.
Then $\dim_\C(A \otimes \kappa(x))=k+a$
and $\dim_\C(B \otimes \kappa(x))=r-k+b$.
Using the exact sequence
    $$0 \to \Tor^1(A,\kappa(x)) \to B\otimes \kappa(x) \to E\otimes \kappa(x) \to  A\otimes \kappa(x) \to 0\,.$$
and $\dim (\Tor^1(A,\kappa(x))) = \dim (\Tor^1(\m_x^{\oplus a},\kappa(x))) = 2a$,
we get $b=a$.
\end{proof}

\begin{remark}\label{remark type of a module}
    Let $M$ be a torsion-free $\O_{C,x}$-module of rank $k$.
    Using \cite[Part 8, Proposition 2]{Ses82}, 
    we can similarly make sense of the type of the module $M$.
    Precisely, if $M \cong \O_{C,x}^{\oplus k-a} \oplus \m_x^{\oplus a}$ then the type of $M$ is defined as $a$.
    As the minimal number of generators of $\m_x$ is 2,
    the minimal number of generators of $M$, in this case, is 
    $(k-a)+2a = k+a$.
    Moreover, if 
    $$ 0 \to M \to \O_{C,x}^{\oplus m} \to N \to 0$$
    be a short exact sequence of torsion-free $\O_{C,x}$-modules
    then type of $M$ and type of $N$ are equal.
    The proof is exactly the same as that of Lemma \ref{lemma type of kernel}.
\end{remark}

For a torsion sheaf $\tau$ on $C$,  we define the socle of $\tau_x$ as
$\soc(\tau_x) := \Hom_C(\kappa(x),\tau)$.
\begin{lemma}\label{lemma dimension of socle}
    Let $0<k<r$.
    Let $A$ and $B$ be two torsion-free sheaves of
    rank $r-k$.
    Let type of $A$ be $a$ and type of $B$ be $b$.
    Let $\tau$ be a torsion sheaf on $C$ supported at the node $x$, which fits into a short exact sequence
    \begin{equation}\label{ 0 to B to A to tau}
         0 \to B \to A \to \tau \to 0\,.
    \end{equation}
    Then $\dim(\soc(\tau)) \leqslant r-k+2a+b$.
\end{lemma}
\begin{proof}
Localizing the short exact sequence 
\eqref{ 0 to B to A to tau} at the node $x$,
we get the short exact sequence
$$0 \to B_x \to A_x \to \tau \to 0\,. $$
As $A_x = \O_{C,x}^{\oplus r-k-a} \oplus \m_x^{\oplus a}$,
there is a surjection $\vp:\O_{C,x}^{\oplus r-k+a} \to A_x$.
Let $\psi :\O_{C,x}^{\oplus r-k+a} \to \tau$ denote the composition 
of $\vp$ and the quotient $A \to \tau \to 0$.
Let $K$ be the kernel of $\vp$ and $M$ be the kernel $\psi$.
So we have the following commutative diagram of $\O_{C,x}$-modules, whose rows and columns are exact
\begin{equation}\label{}
    \begin{tikzcd}
    & 0 \arrow{d}{} & 0 \arrow{d}{}\\
    &K \ar[r,-,double equal sign distance,double] \arrow{d}{} & K \arrow{d}{}\\
    0 \arrow{r}{} & M \arrow{r}{} \arrow{d}{} & \O_{C,x}^{\oplus r-k+a} \arrow{d}{\vp} \arrow{r}{\psi} & \tau \ar[d,-,double equal sign distance,double] \arrow{r}{} & 0 \\
    0 \arrow{r}{} & B_x \arrow{r}{} \arrow{d}{} & A_x \arrow{r}{} \arrow{d}{} & \tau \arrow{r}{} & 0 \\
    & 0  & 0 \,. \\
    \end{tikzcd}
\end{equation}
The sheaf $B$ is of rank $r-k$ and type $b$.
So $B_x$ can be generated by $r-k+b$ elements.
By Remark \ref{remark type of a module}, $K$ is of rank $a$ and type $a$.
So $K$ can be generated by $2a$ elements.
From the left vertical exact sequence, it is clear that 
$M$ can be generated by $r-k+b+2a$ elements.

We note that the socle of $\tau$ depends only on the module structure over the completion $\widehat{\O_{C,x}}$,
which is isomorphic to $\C[[X,Y]]/(XY)$.
So for the purpose of computing $\dim(\soc(\tau))$,
we can replace $C$ by the curve $C'$ defined by the vanishing of the polynomial $\wt X \wt Y$
in the plane $S:=\mb P^2_{\C} = \text{Proj}(\C[\wt X,\wt Y,\wt Z])$.
Thus we have a short exact sequence
$$ 0 \to \O_{S,x} \to \O_{S,x} \to \O_{C',x} \to 0\,.$$
Using this we can construct a commutative diagram whose 
rows and columns are exact
\begin{equation}\label{}
    \begin{tikzcd}
    & 0 \arrow{d}{} & 0 \arrow{d}{}\\
    & \O_{S,x}^{\oplus r-k+a} \ar[r,-,double equal sign distance,double] \arrow{d}{} & \O_{S,x}^{\oplus r-k+a} \arrow{d}{}\\
    0 \arrow{r}{} & N \arrow{r}{} \arrow{d}{} & \O_{S,x}^{\oplus r-k+a} \arrow{d}{} \arrow{r}{} & \tau \ar[d,-,double equal sign distance,double] \arrow{r}{} & 0 \\
    0 \arrow{r}{} & M \arrow{r}{} \arrow{d}{} & \O_{C',x}^{\oplus r-k+a} \arrow{r}{\psi} \arrow{d}{} & \tau \arrow{r}{} & 0 \\
    & 0  & 0 \,. \\
    \end{tikzcd}
\end{equation}
Similarly, from this diagram, we get that $N$
can be generated by $(r-k+b+2a) + (r-k+a) = 2(r-k)+3a+b$ elements.
We view the surjection $\O_{S,x}^{\oplus r-k+a} \to \tau \to 0$ as a point in the Quot scheme of torsion quotients of $\O_{S}^{\oplus r-k+a}$
on the surface $S$.
Then using \cite[Lemma 2]{EL99}, we conclude that 
$\dim (\soc(\tau)) \leqslant (2(r-k)+3a+b) - (r-k+a) = r-k+2a+b$.
\end{proof}

Recall that for a local ring $R$ and two $R$-modules $M$ and $N$, we have the $R$-modules $\Ext^i_R(M,N)$.
When $N$ has finite length, the module $\Ext^i_R(M,N)$ has finite length.
We denote the length of $\Ext^i_R(M,N)$ by $\ext^i_R(M,N)$.
Let us take $R$ to be the local ring $\O_{C,x}$.
Let $\tau$ be a torsion sheaf over $C$ supported at the point $x$ and $B$ be any sheaf on $C$.
Then we have defined the numbers $\ext^i_R(B_x,\tau)$.
On the other hand, $\ext^i(B,\tau)$ is the dimension of the $\C$-vector space $\Ext^i_C(B,\tau)$.
In the next lemma, we will use the fact that 
$\ext^i(B,\tau) = \ext^i_R(B_x,\tau)$ for all $i \geqslant 0$.
This can be seen as follows.
Consider the local-to-global spectral sequence
$$ E_2^{p,q}=H^p(C,\EXT^q(B,\tau)) \implies \Ext_C^{p+q}(B,\tau)\,.$$
As $\tau$ is supported at the point $x$, the sheaves $\EXT^i(B,\tau)$ are also supported at $x$.
Consequently $E_2^{p,q}=H^p(C,\EXT^q(B,\tau)) =0$ for any $p \neq 0$ and for any $q$.
It follows that 
$$E_{\infty}^{p,q}=
\begin{cases}
    H^0(C,\EXT^q(B,\tau)) \quad & \text{if }p=0\\
    0 & \text{otherwise} 
\end{cases}
$$
For any $i\geqslant 0$, we have 
$\Ext^i(B,\tau)=E^{0,i}_{\infty}=H^0(C,\EXT^i(B,\tau))$.
On the other hand, $\EXT^i(B,\tau)$ is supported at the point $x$, so $H^0(C,\EXT^i(B,\tau)) = \Ext^i_R(B_x,\tau)$.
This proves that $\ext^i(B,\tau) = \ext^i_R(B_x,\tau)$ for any $i\geqslant 0$.

\begin{lemma}\label{lemma hom-ext1 general}
    Let $0<k<r$.
    Let $E$ be a locally free sheaf of rank $r$ and degree $e$
    and let $A$ be a sheaf of rank $k$ and degree $d$.
    Assume that 
    $0 \to B \to E \to A \to 0$ is a short exact sequence of sheaves on $C$.
    Then we have
    $$ \hom(B,A) - \ext^1(B,A) \geqslant dr-ke+k(r-k)(1-g)-(r-k)(2r+k) \,.$$
\end{lemma}
\begin{proof}
    Let $\tau$ be the torsion subsheaf of $A$, that is,
    the maximal 0-dimensional subsheaf of $A$. 
    Let $A'$ be the torsion-free quotient $A/\tau$,
    so that we have a short exact sequence
    \begin{equation}\label{equation 1st ses for tau}
    0 \to \tau \to A \to A' \to 0\,.
    \end{equation}
    Let the length of $\tau$ be $\delta$.
    We separate the torsion sheaf $\tau$ in two parts.
    Let $\tau_x$ denote the torsion at the node $x$ and $\tau_o$ denote the torsion outside $x$.
    Thus, we have $\tau = \tau_x \oplus \tau_o$.
    Let length$(\tau_x)=\delta_x$ and length$(\tau_o)=\delta_o$.
    So $\delta_x + \delta_o = \delta$.
    Let $B'$ be the kernel of the surjection $E \to A'\to 0$.
    By Lemma \ref{lemma type of kernel}, type of $B'$ is same as type of $A'$.
    Also, there is a short exact sequence
    \begin{equation}\label{equation 2nd ses for tau}
    0 \to B \to B' \to \tau \to 0\,.
    \end{equation}
    Applying $\Hom(B,\_)$ to the short exact sequence \eqref{equation 1st ses for tau},
    we get
    \begin{equation}\label{equation hom-ext separation}
        \begin{aligned}
        \hom&(B,A) - \ext^1(B,A) \\
        & \geqslant \hom(B,A') - \ext^1(B,A')  +
        \hom(B,\tau_o) - \ext^1(B,\tau_o) + 
        \hom(B,\tau_x) - \ext^1(B,\tau_x)\,.
        \end{aligned}
    \end{equation}
We compute the terms $(\hom(B,A')-\ext^1(B,A'))$, $(\hom(B,\tau_o)-\ext^1(B,\tau_o))$ and $(\hom(B,\tau_x)-\ext^1(B,\tau_x))$.
Let the type of $A'$ be $a'$ and the type of $B$ be $b$.
Degree of $A'$ is $d-\delta$ and degree of $B$ is $e-d$.
By Lemma \ref{Lemma hom(B,A) - ext1(B,A) torsionfree}, we get that
\begin{equation}\label{equation hom -ext1 torsion-free}
 \hom(B,A') - \ext^1(B,A') = dr-ke+k(r-k)(1-g)-a'b- \delta(r-k) \,.   
\end{equation}
As $B$ is free on the support of $\tau_o$, 
$\hom(B,\tau_o) = (r-k)\delta_o$ and $\ext^1(B,\tau_o) = 0$.
So we have
\begin{equation}\label{equation hom-ext1 torsion outside node}
\hom(B,\tau_o) - \ext^1(B,\tau_o) = (r-k)\delta_o \,.
\end{equation}
Now we compute $\hom(B,\tau_x)-\ext^1(B,\tau_x)$.
Let $R$ denote the local ring $\O_{C,x}$ and $\m_x$ denote the maximal ideal of $R$.
By the discussion preceding the lemma, we have 
$\hom(B,\tau_x) - \ext^1(B,\tau_x) =
    \hom_R(B_x,\tau_x) - \ext^1_R(B_x,\tau_x)$.
Thus
\begin{equation}\label{equation hom-ext1 torsion at node}
\begin{aligned}
    \hom(B,\tau_x) - \ext^1(B,\tau_x) &= 
    \hom_R(B_x,\tau_x) - \ext^1_R(B_x,\tau_x) \\
    &= \hom_R(R^{\oplus (r-k-b)} \oplus \m^{\oplus b},\tau_x) - \ext^1_R(R^{\oplus(r-k-b)} \oplus \m^{\oplus b},\tau_x)\\
    &= (r-k-b)[\hom_R(R,\tau_x)-\ext^1_R(R,\tau_x)] \\
    & \quad + b[\hom_R(\m_x,\tau_x)-\ext^1_R(\m_x,\tau_x)] \\
    & = (r-k-b)\delta_x + b[\hom_R(\m_x,\tau_x)-\ext^1_R(\m_x,\tau_x)]\,.
\end{aligned}    
\end{equation}

Let $\widehat R$ denote the completion of $R$ with respect to the maximal ideal $\m_x$.
As $\tau$ is also an $\widehat R$-module, we have
$$\hom_R(\m_x,\tau_x)-\ext^1_R(\m_x,\tau_x) = \hom_{\widehat R}(\widehat \m_x,\tau_x)-\ext^1_{\widehat R}(\widehat \m_x,\tau_x) \,. $$
We know that $\widehat R$ is isomorphic to $\C[[X,Y]]/(XY)$.
So we have a projective resolution of $\widehat \m_x$,
$$ \dots \to \widehat R^2 \xrightarrow{\begin{pmatrix}
X & 0 \\
0 & Y
\end{pmatrix}} \widehat R^2 \xrightarrow{\begin{pmatrix}
Y & 0 \\
0 & X
\end{pmatrix}} \widehat R^2 
\longrightarrow \widehat \m_x \to 0 \,.$$
Applying $\Hom(\_,\tau_x)$ to the complex obtained by dropping $\widehat \m_x$, we get the complex
$$ \dots \xleftarrow{} \tau_x^2 \xleftarrow{\phi_2 = \begin{pmatrix}
X & 0 \\
0 & Y
\end{pmatrix}}
\tau_x^2 \xleftarrow{\phi_1 = \begin{pmatrix}
Y & 0 \\
0 & X
\end{pmatrix}}
\tau_x^2 \xleftarrow{} 0 \,.
$$
Calculating the dimension of $\ker \phi_i$ and $\im \phi_i$,
we have
$$ \hom_{\widehat R}(\widehat{\m},\tau_x) - \ext^1_{\widehat R}(\widehat{\m},\tau_x) = \dim(\im (\phi_2))\,.$$
Let $\Delta : \tau_x \to \tau_x^{\oplus 2}$ denote the diagonal map 
and let $\phi:= \phi_2 \circ \Delta$ denote the composite
$\tau_x \xrightarrow{\Delta} \tau_x^{\oplus 2}
\xrightarrow{\phi_2} \tau_x^{\oplus 2}$.
Then $\dim (\im (\phi_2)) \geqslant \dim (\im (\phi))$.
The kernel of $\phi$ is the subspace of $\tau_x$ which is annihilated by $\widehat \m_x$ and hence
$\ker(\phi) = \soc(\tau_x)$.
As $\tau$ fits in the short exact sequence \eqref{equation 2nd ses for tau}, using Lemma \ref{lemma dimension of socle},
we get that $\dim(\ker(\phi)) = \dim(\soc(\tau_x)) \leqslant r-k+2a'+b$.
So $\dim (\im(\phi)) \geqslant \delta_x - (r-k+2a'+b)$.
Thus we have 
$$\hom_{\widehat R}(\widehat{\m},\tau_x) - \ext^1_{\widehat R}(\widehat{\m},\tau_x) = \dim(\im (\phi_2)) 
\geqslant \dim(\im (\phi)) \geqslant \delta_x - (r-k+2a'+b)\,.$$
Using \eqref{equation hom-ext1 torsion at node}, we get
\begin{equation}\label{equation hom-ext1 torsion at node second}
    \begin{aligned}
 \hom(B,\tau_x) - \ext^1(B,\tau_x)  
    & = (r-k-b)\delta_x + b[\hom_R(\m_x,\tau_x)-\ext^1_R(\m_x,\tau_x)] \\
   &  \geqslant (r-k-b)\delta_x + b(\delta_x - (r-k+2a'+b))  \\
   & \geqslant (r-k)\delta_x - b(r-k+2a'+b) \,.
    \end{aligned}
\end{equation}
Combining \eqref{equation hom-ext separation}, \eqref{equation hom -ext1 torsion-free}, \eqref{equation hom-ext1 torsion outside node} and \eqref{equation hom-ext1 torsion at node second},
we have
\begin{align*}
    &\hom(B,A) - \ext^1(B,A) \\
    & \geqslant [dr-ke+k(r-k)(1-g)-a'b- \delta(r-k)] + [(r-k)\delta_o]+ [(r-k)\delta_x - b(r-k+2a'+b)]\\
    &= dr-ke+k(r-k)(1-g) - b(r-k+3a'+b)\,.
\end{align*}
Note that $a'\leqslant k$ and $b \leqslant (r-k)$.
Thus we have
$$ \hom(B,A) - \ext^1(B,A) \geqslant dr-ke+k(r-k)(1-g) - (r-k)(2r+k)\,.$$
\end{proof}

    We generalize the results of this section for curves with more than one node.
    Let $C$ be an integral projective curve over $\C$ of arithmetic genus $g \geqslant 2$, with $n$ nodes and smooth elsewhere.
    We name the nodal points $x_1, \dots x_n$.
    \begin{definition}
    Let $F$ be a coherent sheaf on $C$ of rank $k$ such that the stalks $F_{x_1}, \dots, F_{x_n}$ at the nodes are torsion-free.
    Using \cite[Part 8, Proposition 2]{Ses82}, there are integers $a_i$ satisfying $0 \leqslant a_i \leqslant k$
    such that $F_{x_i} \cong \O_{C,x_i}^{\oplus k-a_i} \oplus \m_{x_i}^{\oplus a_i}$.
    We define type$_{x_i}(F)$, the type of $F$ at $x_i$, to be  $a_i$ and we define
    the type of $F$ to be the $n$-tuple $(a_1,\dots,a_n)$.
    \end{definition}
    Using \cite[Lemma 2.5(B)]{bh06}, the statement for Lemma \ref{Lemma hom(B,A) - ext1(B,A) torsionfree} generalizes to the following:
    Let $A$ and $B$ be torsion-free sheaves on $C$ of rank $r_A, r_B$ and degree $d_A, d_B$ respectively.
    Then
    $$ \hom(B,A) - \ext^1(B,A) 
    = r_Bd_A - r_Ad_B + r_Ar_B(1-g) -\sum_{i} \text{type}_{x_i}(A) \cdot \text{type}_{x_i}(B)\,.$$
    Lemma \ref{lemma type of kernel} and \ref{lemma hom-ext1 general} generalize to the following.
    \begin{lemma}\label{lemma hom-ext1 general multi-node}
    Let $0<k<r$.
    Let $E$ be a locally free sheaf of rank $r$ and degree $e$
    and let $A$ be a sheaf of rank $k$ and degree $d$.
    Assume that 
    $0 \to B \to E \to A \to 0$ is a short exact sequence of sheaves on $C$.
    Then we have
    \begin{align*}
        \textnormal{(i) } & \, \textnormal{type}_{x_i}(A)= \textnormal{type}_{x_i}(B) \quad \textnormal{for each node } x_i.\\
        \textnormal{(ii) } & \hom(B,A) - \ext^1(B,A) \geqslant dr-ke+k(r-k)(1-g)-n(r-k)(2r+k) \,.
    \end{align*}
    \end{lemma}
    \begin{proof}
    Part (i) follows directly from Lemma \ref{lemma type of kernel}.
    The proof of part (ii) is similar to the proof of Lemma \ref{lemma hom-ext1 general}.
    Let $\tau$ be the torsion subsheaf of $A$.
    Further, let $\tau_{x_i}$ denote the torsion of $A$ at the node $x_i$
    and $\tau_o$ denote the torsion of $A$ outside the nodes,
    so that we have $\tau = \left(\oplus_{i=1}^n \tau_{x_i} \right) \oplus \tau_o$.
    The inequality \eqref{equation hom-ext separation} becomes
        \begin{align*}\label{equation hom-ext separation multi-node}
        &\hom(B,A) - \ext^1(B,A) \\
        & \, \geqslant \hom(B,A') - \ext^1(B,A')  +
        \hom(B,\tau_o) - \ext^1(B,\tau_o) + 
        \sum_{i=1}^n\left(\hom(B,\tau_{x_i}) - \ext^1(B,\tau_{x_i})\right)\,.
        \end{align*}
    Calculating each term separately as in the proof of Lemma \ref{lemma hom-ext1 general}, we have 
    \begin{align*}
    \hom(B,A) - \ext^1(B,A) & \geqslant dr-ke+k(r-k)(1-g)- \sum_{i=1}^n (r-k)(2r+k)\\
    & \geqslant dr-ke+k(r-k)(1-g)-n(r-k)(2r+k) \,.
    \end{align*}
    \end{proof}

\color{black}
\section{Dimensions of some subsets of the Quot scheme}
\label{section dimension of subsets}
From here on, we will assume $C$ to be an integral projective curve over $\C$ of arithmetic genus $g > n$, with $n$ nodes at the points $x_1, \dots x_n$ and smooth elsewhere.
Let $E$ be a vector bundle on $C$ of rank $r$ and degree $e$.  

\begin{proposition}\label{proposition dimension bound}
    Every irreducible component $\mc W$ of the Quot scheme $Q_{k,d}(E)$ has dimension at least
    $$dr-ke+k(r-k)(1-g) - n(r-k)(2r+k)\,.$$
\end{proposition}
\begin{proof}
Let $\mc W$ be an irreducible component of $Q_{k,d}(E)$.
Let $[q:E \to F]$ be a closed point of $\mc W$ and
$S_F$ denote the kernel of $q$.
By \cite[Proposition 2.2.8]{HL}, we get the following bound for
the dimension of $Q_{k,d}(E)$ at the point $[q : E \to F]$,
    \begin{equation}\label{dimension bound}
    \dim_{[q]}Q_{k,d}(E) \geqslant  \hom(S_F,F) - \ext^1(S_F,F)\,.
    \end{equation}
Now using Lemma \ref{lemma hom-ext1 general multi-node}\,(ii),
we have 
$$\dim_{[q]}Q_{k,d}(E) \geqslant \hom(S_F,F) - \ext^1(S_F,F) \geqslant dr-ke+k(r-k)(1-g)-n(r-k)(2r+k) \,.$$
Since this is true for any closed point in $\mc W$, it follows that $\mc W$ has dimension at least
$dr-ke+k(r-k)(1-g)-n(r-k)(2r+k)$.
\end{proof}

\noindent
Let us define the following subsets of $Q_{k,d}(E)$.
\begin{align*}
    \nonumber Q_{k,d}^{\tf}(E) := &\{[q: E\to F] \in Q_{k,d}(E) : F \textrm{ is torsion-free }\} \, ,\\
    \nonumber 
    Q_{k,d}^{{\tf},(a_1,\dots,a_n)}(E) := & \{[q: E\to F] \in Q_{k,d}(E) : F \textrm{ is torsion-free of type $(a_1,\dots,a_n)$}\} \\
    \label{definition Q^tf,a locus}
    = & \{[q: E\to F] \in Q_{k,d}(E) : F \textrm{ is torsion-free }\\
    & \phantom{\{[q: E\to F] \in Q_{k,d}(E) :} \text{ and } \dim_\C(F \otimes \kappa(x_i)) = k+a_i \quad \forall i \} \,.
\end{align*}
The set $Q_{k,d}^{\tf}(E)$ is open in $Q_{k,d}(E)$ by \cite[Theorem 2.3.1]{HL}.
The sets $\q{(a_1,\dots,a_n)}(E)$ are locally closed subsets of $Q_{k,d}^{{\tf}}(E)$ using Semicontinuity theorem (\cite[Chapter III, Theorem 12.8]{Ha}). 
We endow these subsets with the reduced induced subscheme structure.

Let $\pi : \widetilde{C} \to C$ be the normalization of $C$
and let $\wt g$ be the arithmetic genus of $\wc$. 
From the exact sequence of cohomology associated to the short exact sequence 
\begin{equation}\label{ses1}
0 \to \O_C \to \pi_*\O_{\wc} \to \oplus_{i=1}^n\kappa(x_i) \to 0  \, ,  
\end{equation}
one can see that $\widetilde{g}=g-n >0$.
\begin{lemma}\label{lemma degree of pullback}
   Let $F$ be a torsion-free sheaf on $C$ of rank $k$ and type $(a_1,\dots,a_n)$. Then $$\deg(\pi^*F) = \deg(F) + \sum_{i=1}^n a_i\,. $$
   In particular, for any torsion-free sheaf $F$ on $C$ of rank $k$, we have $\deg(\pi^*F) \leqslant \deg(F) + nk$.
\end{lemma}
\begin{proof}
Tensoring the short exact sequence \eqref{ses1}
with $F$, we get a short exact sequence
\begin{equation}\label{ses2}
    0 \to F \to \pi_*\pi^*F \to  \oplus_{i=1}^n (F \vert_{x_i}) \to 0\,.
\end{equation}
The lemma follows from this using the Riemann-Roch theorem.
\end{proof}

\begin{proposition}\label{proposition f_a map}
For each type $(a_1,\dots,a_n)$, there exists a map
$$f_{(a_1,\dots,a_n)} : \q{(a_1,\dots,a_n)}(E) \longrightarrow \Quot_{\wc/\C}(\pi^*E,k,d+\sum_{i=1}^n a_i)\, ,$$ which is injective on closed points.
\end{proposition}
\begin{proof}
    We prove the proposition for $n=1$ case. 
    Proof for the general case is similar.
    Let us fix a type $a$ with $0\leqslant a \leqslant k$.
    Let $p_C : C \times Q_{k,d}(E) \to C$ be the projection map. 
    There is a universal quotient $p_C^*E \to \F\to 0$ 
    on $C \times Q_{k,d}(E)$, $\F$ is flat over $Q_{k,d}(E)$.
    Let $\F_a$ denote the restriction of $\F$ on the subset $\q{a}(E)$.
    Let $\wt \pi : \wc \times \q{a}(E) \to C \times \q{a}(E)$
    denote the map $\pi \times \id_{\q{a}(E)}$.
    The pullback $\wt \pi^*\F_a$ is flat over $\q{a}(E)$.
    So the quotient
    \begin{equation}\label{pullback quotient}
        \wt \pi^* p_C^*E \longrightarrow \wt \pi^*\F_a \longrightarrow 0\,.
    \end{equation} 
    induces a map
    $$f_a : \q{a}(E) \longrightarrow \Quot_{\wc/\C}(\pi^*E,k,d+a)\, ,$$
    as we have seen that $\deg(\pi^*F) = \deg(F)+a$ in Lemma \ref{lemma degree of pullback}.

    To show that the map $f_a$ is injective on closed points,
    we start with two closed points $[\vp_1: E\to F_1]$ and $[\vp_2:E \to F_2]$ in $\q{a}(E)$ such that $[\pi^*\vp_1] = [\pi^*\vp_2]$
    in $\Quot_{\wc/\C}(\pi^*E,k,d+a)$.
    Let $S_i = \ker(\vp_i)$ and $K_i = \ker(\pi^*\vp_i)$.
    Then we have $K_1=K_2$.
    We have the inclusions $j_E: E \to \pi_*\pi^*E$ and
    $j_1: F_1 \to \pi_*\pi^*F_1$ as in the short exact sequence \eqref{ses2}.
    Moreover, the following diagram is commutative,
    \begin{equation}
        \begin{tikzcd}
        E \ar[d,hook,"j_E"] \arrow{r}{\vp_1} & F_1 \ar[d,hook,"j_1"]\\
        \pi_*\pi^*E \arrow{r}{\pi_*\pi^*\vp_1} & \pi_*\pi^*F_1  \,.
    \end{tikzcd}
    \end{equation}
    So a local section $s$ of $E$ belongs to $S_1$ if and only if $j_E(s)$ lands inside the subsheaf $\pi_*K_1 \subset \pi_*\pi^*E$.
    It follows that $S_1 = E \cap \pi_*K_1$.
    For the same reason, we also have $S_2 = E \cap \pi_*K_2$.
    As $K_1=K_2$, we have $S_1=S_2$.
    This proves the injectivity of $f_a$ on closed points.
\end{proof}

We will use the map $f_{(a_1,\dots,a_n)}$ to bound the dimension of some subsets of the Quot scheme $Q_{k,d}(E)$ using bounds of related subsets in the Quot schemes $\Quot_{\wc/\C}(\pi^* E,k,d')$.
We start by computing the dimension bound for a subset of $\Quot_{\wc/\C}(\pi^* E,k,d')$.
The following lemma is contained in the proof of \cite[Lemma 6.3]{PR03}. 
However, we give a detailed proof for the convenience of the reader.
For a sheaf $G$ on $\wc$ and for each rank $k < \rk(G)$, define
$$ m(G,k) = \min_{\rank(F)=k} \left\{ \deg(F) : F \textrm{ is a quotient of } G\right\}\,.$$
The function $m(\underline{\hspace{.3cm}},k)$ is lower semi-continuous in a flat family of sheaves.

For ease of notation, let us denote $\pi^*E$ by $\wt E$ for the rest of this section.
Let $d'$, $d_0$ be integers and let $k_0$ be an integer such that $0<k_0<k$.
Let $\wt Z$ be the subset of $\Quot_{\wc/\C}(\wt E,k,d')$ defined as follows,
\begin{equation}\label{definition wt Z}
\wt Z = \{ [\wt E \to G] \in \Quot_{\wc/\C}(\wt E,k,d') : \exists \textrm{ a quotient } G \to G_0 \textrm{ of rank }k_0 \textrm{ and degree } d_0 \}\,.
\end{equation}
We give $\wt Z$ the reduced induced subscheme structure.
\begin{lemma}\label{lemma dimension wt Z}
There exist constants $C_1(E,k,k_0,d_0)$ and $\alpha_1(E,k,k_0,d_0)$ such that
we have
$$\dim \wt Z \leqslant  d'(r-k_0) + C_1 \quad \quad \text{ for any } \quad d' \geqslant \alpha_1\,.$$
\end{lemma}

\begin{proof}
    Let $A$ denote the Quot scheme $\Quot_{\wc/\C}(\wt E,k_0,d_0)$.
    Let $\S_0$ denote the universal kernel on $\wc \times A$.
    Consider the relative Quot scheme 
    $$ D:= \Quot_{\wc \times A/A}(\S_0,k-k_0,d'-d_0)\,.$$
    A closed point of $D$ corresponds to a pair $([q_0],[\vp])$ where
    $[q_0 : \wt E\to F_0] \in A$ and
    $\vp: \ker (q_0) \to B$ is a quotient of rank $k-k_0$ and degree $d'-d_0$.
    Given such a point, we can construct a quotient $q:\wt E \to F$ 
    using the following push-out diagram
    $$ 
    \begin{tikzcd}
        0 \ar{r}{} & \ker(q_0) \ar{r}{} \ar{d}{\vp} & \wt E \ar{r}{q_0} \ar{d}{q} & F_0 \ar[d,-,double equal sign distance,double] \ar{r}{} & 0 \phantom{\,.}\\
        0 \ar{r}{} & B \ar{r}{} & F \ar{r}{} & F_0 \ar{r}{} & 0 \,.
    \end{tikzcd}
    $$
    Clearly, the degree of $F$ is $d'$ 
    and so $q$ corresponds to a point of $\Quot_{\wc/\C}(\wt E,k,d')$.
    Doing the same construction with the universal families,
    we get a map of schemes 
    $$ g : D \to \Quot_{\wc/\C}(\wt E,k,d')$$
    which sends the point $([q_0],[\vp])$ to the point $[q]$.
    Given any point $[q:\wt E \to F]$ in $\wt Z$, 
    we have a quotient $[q_0 :\wt E \to F_0] \in A$ (by definition of the locus $\wt Z$)
    and we get a quotient 
    $\ker q_0 \to B \to 0$ of rank $k-k_0$ and degree $d'-d_0$. 
    This shows that $g$ maps onto the subset $\wt Z$.
    Consequently, we have the following upper bound
    \begin{equation}\label{equation dim wt Z}
    \begin{aligned}
    \dim \wt Z & \leqslant \dim D 
    \leqslant \dim A + \max_{q_0 \in A}\left\{\dim \left(\Quot_{\wc/\C}(\ker q_0,d'-d_0,k-k_0)\right)\right\}\,.
    \end{aligned}
    \end{equation}
Using \cite[Theorem 4.1]{PR03} we have, for all $d'-d_0 \geqslant m(\ker q_0,k-k_0)$,
\begin{equation*}
\dim\left(\Quot_{\wc/\C}(\ker q_0,d'-d_0,k-k_0)\right) \leqslant
(k-k_0)(r-k) +  (d'-d_0-m(\ker q_0,k-k_0))(r-k_0) \,.
\end{equation*}
As the collection $\{\ker q_0: q_0 \in A\}$ is a bounded family
and $m(\underline{\hspace{.3cm}},k-k_0)$ is lower semicontinuous,
it follows that the set 
$\{m(\ker q_0,k-k_0) : q_0 \in A\}$ is finite.
So there exist constants $M'(E,k)$ and $M''(E,k)$ such that 
$$ M' \leqslant m(\ker q_0,k-k_0) \leqslant  M''
\quad \textrm{ for all } [q_0] \in A.$$
It follows that, for all  $d' \geqslant M''+d_0$, we have
\begin{align*}
\dim\left(\Quot_{\wc/\C}(\ker q_0,d'-d_0,k-k_0)\right) \leqslant
(k-k_0)(r-k) + & (d'-d_0-M')(r-k_0)\,.
\end{align*}
Consequently, from \eqref{equation dim wt Z} we have
\begin{align*}
   \dim \wt Z &\leqslant \dim A + \max_{q_0 \in A}\left\{\dim \left(\Quot_{\wc/\C}(\ker q_0,d'-d_0,k-k_0)\right)\right\} \\
   &\leqslant \dim(\Quot_{\wc/\C}(\wt E,k_0,d_0)) +(k-k_0)(r-k) + (d'-d_0-M')(r-k_0) \\
   & \leqslant d'(r-k_0) + \dim(\Quot_{\wc/\C}(\wt E,k_0,d_0)) + (k-k_0)(r-k) -(d_0+M')(r-k_0)  \,.
   \end{align*}
We define $\alpha_1 := M'' +d_0$ and $C_1:= \dim(\Quot_{\wc/\C}(\wt E,k_0,d_0)) + (k-k_0)(r-k) -(d_0+M')(r-k_0)$.
This proves the lemma.
\end{proof}

\begin{lemma}\label{lemma dimension Y}
Let $d_0$ and $k_0$ be integers such that $0<k_0<k$.
Let $Y_{d_0,k_0}$ be the subset of $Q^{\tf}_{k,d}(E)$ defined as,
\begin{equation}\label{equation definition Y}
    Y_{d_0,k_0} = \{ [E \to F] \in Q^{\tf}_{k,d}(E): \pi^*F \textrm{ has a quotient of rank $k_0$ and degree } d_0 \}\,,
\end{equation}
endowed with the reduced induced subscheme structure.
Then there exist constants $C_2(E,k,k_0,d_0)$ and $\alpha_2(E,k,k_0,d_0)$ such that 
$$\dim Y_{d_0,k_0} \leqslant  d(r-k_0) + C_2 \quad \quad \text{ for any } \quad d \geqslant \alpha_2\,.$$
\end{lemma}
\begin{proof}
    For the sake of notation, we denote $Y_{d_0,k_0}$ by $Y$ in this proof.
    For any type $(a_1,\dots,a_n)$, let $Y^{(a_1,\dots,a_n)} := Y \cap \q{(a_1,\dots,a_n)}(E)$.
    Then $Y$ is disjoint union of the finitely many subsets $Y^{(a_1,\dots,a_n)}$ and it is enough to bound the dimension of $Y^{(a_1,\dots,a_n)}$'s.
    Fix a type $(a_1,\dots,a_n)$ and let $a = \sum_{i=1}^n a_i$.
    Consider the map defined in Proposition \ref{proposition f_a map},
    $$f_{(a_1,\dots,a_n)} : \q{(a_1,\dots,a_n)}(E) \to\Quot_{\wc/\C}(\wt E,k,d+a)\,.$$
    Consider the subset $\wt Z$ of  $\Quot_{\wc/\C}(\wt E,k,d+a)$
    as defined in \eqref{definition wt Z}, taking $d'=d+a$.
    It is clear that the image $f_{(a_1,\dots,a_n)}(Y^{(a_1,\dots,a_n)})$ is contained in $\wt Z$. 
    Using Lemma \ref{lemma dimension wt Z}, we get constants $C_1$ and $\alpha_1$ such that 
    $$\dim \wt Z \leqslant  (d+a)(r-k_0) + C_1 \quad \text{ for } d+a \geqslant \alpha_1 \,.$$
    As $f_{(a_1,\dots,a_n)}$ is injective on points, we have the following inequality
    $$ \dim Y^{(a_1,\dots,a_n)} = \dim (f_{(a_1,\dots,a_n)}(Y^{(a_1,\dots,a_n)})) \leqslant \dim \wt Z \leqslant (d+a)(r-k_0) + C_1 
    \quad \text{ for } d \geqslant \alpha_1 \,.$$
    Taking $\alpha_2 := \alpha_1$ and $C_2:= nk(r-k_0)+C_1$,
    we have 
    $$\dim Y \leqslant  d(r-k_0) + C_2 \text{ for } \quad d \geqslant \alpha_2 \,.$$
    This proves the lemma.
\end{proof}

\begin{lemma}\label{lemma dimension Ymu}
Given a number $\mu_0$, 
define the subset $Y_{\mu_0}$ of $Q^{\tf}_{k,d}(E)$ as 
$$Y_{\mu_0} = \{ [E \to F] \in Q^{\tf}_{k,d}(E): \mu_{\min}(F) < \mu_0 \}\,,$$
endowed with the reduced induced subscheme structure.
Then there exist constants $C_3(E,k,\mu_0)$ and $\alpha_3(E,k,\mu_0)$ such that 
$$\dim Y_{\mu_0} \leqslant  d(r-1) + C_3  \quad \quad \text{ for any } \quad d \geqslant \alpha_3\,.$$
\end{lemma}
\begin{proof}
    Let $[q:E \to F] \in Y_{\mu_0}$ be a closed point.
    Let $F \to F'\to 0$ denote the minimal destabilizing quotient of $F$.
    Note that $F'$ is torsion-free and $\mu(F') = \mu_{\min}(F) < \mu_0$.
    Again $F'$ is a quotient of $E$, so $\mu(F') \geqslant \mu_{\min}(E)$.
    Thus $\mu_{\min}(E) \leqslant \mu(F') < \mu_0$.
    As $1 \leqslant \rk(F') \leqslant k-1$, we have 
    $\mu_{\min}(E) \leqslant \deg(F') < (k-1)\mu_0$.
    Using Lemma \ref{lemma degree of pullback}, we see that
    $\mu_{\min}(E) \leqslant \deg(\pi^* F') < (k-1)\mu_0+nk$.
    In particular, there are only finitely many choices for the degree of $\pi^* F'$.
    Recall the locus $Y_{d_0,k_0}$ from \eqref{equation definition Y}.
    As $\pi^*F'$ is a quotient of $\pi^*F$, 
    the point $[q:E \to F]$ is in the locus 
    $$Y_{\deg(\pi^* F'),\rank(\pi^*F')}\,.$$
    Thus, we can write $Y_{\mu_0}$ as a union of finitely many subsets 
    $Y_{d_0,k_0}$ with $1 \leqslant k_0 \leqslant k-1$.
    Now the lemma follows using Lemma \ref{lemma dimension Y}.
\end{proof}

Before going to the next lemma, let us recall the following result.
\begin{lemma}[Corollary in Appendix, Page 87 of \cite{OS79}]\label{lemma oda seshadri}
Let $A \to B$ be a flat local homomorphism of noetherian local rings.
Let $N$ be a $B$-module of finite type which is $A$-flat and satisfies $\Ext^1(\ol N,\ol B) = 0$, 
where $\ol B=B \otimes_A k$ and $\ol N= N \otimes_A k$ for the residue field $k$ of $A$.
Then $\Hom_B(N, B)$ is $A$-flat and 
$\Hom_B(N,B) \otimes_A k = \Hom_{\ol B}(\ol N,\ol B)$.
\end{lemma}

\begin{lemma}\label{lemma dual isomorphism}
For any degree $d$, there is a map of Quot schemes
$$\V : Q_{k,d}^{\tf}(E) \to Q^{\tf}_{r-k,d-e}(E^\vee)$$
which sends a closed point $[q:E \to F]$ to the point $[E^\vee \to (\ker q)^\vee]$.
Moreover, the map $\V$ is an isomorphism.
\end{lemma}

\begin{proof}
    Let $p_C: C \times Q_{k,d}^{\tf} \to C$ denote the projection map.
    Consider the universal short exact sequence 
    on $C \times Q_{k,d}^{\tf}$ :
    \begin{equation}\label{ses universal ses}
        0 \to \S \to p_C^*E \to \F \to 0\,.
    \end{equation}
    The sheaves $\F$ and $\S$ are flat over $Q_{k,d}^{\tf}$.
    Consider the duals
    \begin{equation*}
     \F^\vee := \HOM(\F,\O_{C \times Q_{k,d}^{\tf}}) \quad \text{ and } \quad
     \S^\vee := \HOM(\S,\O_{C \times Q_{k,d}^{\tf}})\,.
    \end{equation*}
We claim that the sheaf $\F^\vee$ is flat over $Q_{k,d}^{\tf}$.
It is enough to check that the stalk $(\F^\vee)_{(c,q)}$ is flat over the local ring $\O_{Q_{k,d}^{\tf},q}$ for any closed point $(c,[q:E \to F]) \in C \times Q_{k,d}^{\tf}$.
To apply Lemma \ref{lemma oda seshadri}, we need $\Ext^1_{\O_{C,c}}(F_c,\O_{C,c})=0$.
Recall that the stalk $F_c$ is isomorphic to $\O_{C,c}^{\oplus k-a} \oplus \m_c^{\oplus a}$ for some integer $a$ such that $0 \leqslant a \leqslant k$.
So $\Ext^1_{\O_{C,c}}(F_c,\O_{C,c})=0$ follows from \cite[Lemma 2.1(3)]{bh06}.
Using Lemma \ref{lemma oda seshadri}, we conclude that 
\begin{enumerate}[\hspace{.6cm}(i)]
    \item \label{statement flat dual} $\F^\vee$ is $Q_{k,d}^{\tf}$-flat and
    \item \label{statement base change} the natural map $\F^\vee \vert_{C \times \{q\}} \longrightarrow F^\vee $ is an isomorphism for any $[q:E \to F] \in Q_{k,d}^{\tf}$.
\end{enumerate}
Doing the same for $\S$, we get that
\begin{enumerate}[\hspace{.5cm}(i)]
 \setcounter{enumi}{2}
    \item \label{statement flat dual kernel} $\S^\vee$ is $Q_{k,d}^{\tf}$-flat and 
    \item \label{statement base change kernel} the natural map $\S^\vee \vert_{C \times \{q\}} \longrightarrow (\ker q)^\vee $ is an isomorphism for any $[q:E \to F] \in Q_{k,d}^{\tf}$.
\end{enumerate}
Moreover, \cite[Chapter III, Proposition 6.8]{Ha} shows that,
for any $[q:E \to F] \in Q_{k,d}^{\tf}$,
$$\EXT^1_C(F,\O_C) = 0\,.$$
Using \cite[Theorem 1.10(i)]{AK80}, we conclude that 
$$\EXT_{C \times Q_{k,d}^{\tf}}^1(\F,\O_{C \times Q_{k,d}^{\tf}})=0\,.$$
So dualizing the short exact sequence \eqref{ses universal ses}
gives an exact sequence
\begin{equation}\label{ses dual of universal ses}
    0 \to \F^\vee \to (p_C^*E)^\vee \to \S^\vee \to 0\,.
\end{equation}
The quotient $(p_C^*E)^\vee \to \S^\vee \to 0$ defines a map of schemes 
$$ \V : Q_{k,d}^{\tf} \to Q_{r-k,d-e}(E^\vee)\,.$$
To get the pointwise description of $\V$, for any $[q:E \to F] \in Q_{k,d}^{\tf}$, we consider the following diagram in which the vertical arrows are the natural maps,
\begin{equation*}\label{}
    \begin{tikzcd}
    & \F^\vee\vert_{C \times \{q\}} \arrow{r}{} \arrow{d}{\cong} & (p_C^*E)^\vee \vert_{C \times \{q\}} \ar[d,-,double equal sign distance,double] \arrow{r}{} &  \S^\vee\vert_{C \times \{q\}} \arrow{d}{\cong} \ar[r] & 0\\
    0 \arrow{r}{} & F^\vee \arrow{r}{} & E^\vee \arrow{r}{} & (\ker q)^\vee\arrow{r}{} & 0 \,.
    \end{tikzcd}
\end{equation*}
The squares commute by the naturality of vertical maps.
The bottom sequence is exact by \cite[Lemma 2.2(i)]{bh06}.
The vertical maps are isomorphisms by statements \eqref{statement base change} and \eqref{statement base change kernel}.
So the map $\V$ sends the closed point $[q:E \to F]$ to the point $[E^\vee \to (\ker q)^\vee \to 0]$.
It is clear that the image of $\V$ lies within the torsion-free locus $Q^{\tf}_{r-k,d-e}(E^\vee)$.
Thus, we have a map
$$\V : Q_{k,d}^{\tf}(E) \longrightarrow Q^{\tf}_{r-k,d-e}(E^\vee)\,.$$
Similarly, we get a map in the reverse direction
$$\V' : Q^{\tf}_{r-k,d-e}(E^\vee) \longrightarrow  Q_{k,d}^{\tf}(E) \,.$$
To see that $\V' \circ \V = \Id$,
we need to show that the natural map $\F \to \F^{\vee\vee}$ is an isomorphism.
Note that $\F^\vee$ is $Q_{k,d}^{\tf}$-flat and each fiber over $Q_{k,d}^{\tf}$ is torsion-free.
So we can replace $\F$ by $\F^\vee$ in statements \eqref{statement flat dual} and \eqref{statement base change} to get that $(\F^\vee)^\vee$ is $Q_{k,d}^{\tf}$-flat
and the natural map $\F^{\vee \vee} \vert_{C \times \{q\}} \longrightarrow F^{\vee \vee} $ is an isomorphism for any $[q:E \to F] \in Q_{k,d}^{\tf}$.
Moreover, we have the following commutative diagram of natural maps,
\begin{equation*}\label{}
    \begin{tikzcd}
    \F\vert_{C \times \{q\}} \arrow{r}{} \arrow{d}{\cong} & 
    \F^{\vee \vee}\vert_{C \times \{q\}} \arrow{d}{\cong}\\
    F \arrow{r}{} & F^{\vee \vee} \,.
    \end{tikzcd}
\end{equation*}
The map at the bottom is an isomorphism using \cite[Lemma 2.6(i)]{bh06}.
Using flatness of $\F^{\vee\vee}$ over $Q_{k,d}^{\tf}$, we conclude that the natural map $\F \to \F^{\vee \vee}$ is an isomorphism.
Similarly, one can prove that $\V \circ \V' = \Id$.
It follows that $\V$ is an isomorphism.
\end{proof}

\begin{lemma}\label{lemma dimension Xmu}
Given a number $\mu_0$, 
define the subset $X_{\mu_0}$ of $Q^{\tf}_{k,d}(E)$ as follows
$$X_{\mu_0} = \{ [q:E \to F] \in Q^{\tf}_{k,d}(E): \mu_{\max}(\ker q) > \mu_0 \}\,.$$
We endow $X_{\mu_0}$ with the reduced induced subscheme structure.
Then there exist constants $C_4(E,k,\mu_0)$ and $\alpha_4(E,k,\mu_0)$ such that for $d \geqslant \alpha_4$,
$$\dim X_{\mu_0} \leqslant  d(r-1) + C_4 \,.$$
\end{lemma}
\begin{proof}
Let $[q: E \to F]$ be a closed point in $X_{\mu_0}$ and $S_F$ denote the kernel of $q$. 
Then we have $\mu_{\max}(S_F) > \mu_0$ which implies that $\mu_{\min}(S_F^\vee)<-\mu_0$.
We consider the map
$$\V : Q_{k,d}^{\tf}(E) \to Q^{\tf}_{r-k,d-e}(E^\vee)$$
which sends the point $[q: E \to F]$ to the point $[E^\vee \to S_F^\vee]$.
Replacing $E$ by $E^\vee$ in Lemma \ref{lemma dimension Ymu}, 
we consider the subset $Y_{-\mu_0}$ of $ Q^{\tf}_{r-k,d-e}(E^\vee)$, which contains all the points $[E^\vee \to G]$ such that $\mu_{\min}(G) < -\mu_0 $.
It follows that the image $\V(X_{\mu_0})$ is the same as the subset $Y_{-\mu_0}$ in $ Q^{\tf}_{r-k,d-e}(E^\vee)$.
By Lemma \ref{lemma dimension Ymu}, we get constants $C'_3:=C_3(E^\vee,k,-\mu_0)$ and $\alpha_3':=\alpha_3(E^\vee,k,-\mu_0)$
such that $\dim Y_{-\mu_0} \leqslant (d-e)(r-1)+C'_3$ for $d \geqslant \alpha'_3$.
As $\V$ is bijective on points,
we get that for $d \geqslant \alpha'_3$,
$$ \dim X_{\mu_0} = \dim (\V(X_{\mu_0})) \leqslant (d-e)(r-1)+C'_3\,.$$
Taking $\alpha_4:=\alpha'_3$ and $C_4:= -e(r-1)+C'_3$,
the Lemma follows.
\end{proof}

\section{Irreducible components containing stable locally free quotients}
\label{section component stable locally free}
In this section we will show that for large enough $d$, there is a unique irreducible component of $Q_{k,d}(E)$ whose general point corresponds to a stable locally free quotient.
Before proving the theorem we prove a few lemmas which will be required in the proof of the theorem.
\begin{lemma}\label{lemma gg and H1 vanishing of EdualF}
There exists an integer $d_1 = d_1(E,k)$ such that for all stable locally free sheaf $F$ of rank $k$ and degree $d \geqslant d_1$ on $C$, the sheaf
$E^\vee \otimes F$ is generated by global sections and $H^1(E^\vee \otimes F) = 0$.
\end{lemma}
\begin{proof}
    This can be proved exactly as it is proved in the smooth curve case in \cite[Lemma 6.1]{PR03}.
\end{proof}

\begin{lemma}\label{lemma H1 vanishing of S_FdualF}
    If $d \geqslant d_1$, then for any quotient $E \to F \to 0$ where $F$ is a stable locally free sheaf of rank $k$ and degree $d$, we have $\h^1(S_F^\vee \otimes F)=0$ where $S_F$ is the kernel of the quotient $E \to F$.
    As a consequence we have 
    $$\chi(S_F^\vee \otimes F) = h^0(S_F^\vee \otimes F) = dr-ke+k(r-k)(1-g)$$
\end{lemma}
\begin{proof}
    Using lemma \ref{lemma gg and H1 vanishing of EdualF}
    and the short exact sequence $$0 \to F^\vee \otimes F \to E^\vee \otimes F \to S_F^\vee \otimes F \to 0\,,$$ 
    it follows that $\chi(S_F^\vee \otimes F) = h^0(S_F^\vee \otimes F)$. By Riemann–Roch theorem, this is $dr-ke+k(r-k)(1-g)$.
\end{proof}

\begin{lemma}\label{lemma generic surjection}
    Let $F$ be a stable locally free sheaf of rank $k$ and degree $d$ such that the sheaf $\HOM(E,F) = E^\vee \otimes F$ is globally generated and $H^1(E^\vee \otimes F)=0$.
    Then a general homomorphism from $E$ to $F$ is surjective.
\end{lemma}
\begin{proof}
The idea of the proof is taken from the proof of \cite[Lemma 7.12]{gs24}.
Consider the space $ \P(\Hom(E,F)^\vee)$ which parametrizes non-zero maps $E \to F$.
Let $T$ be the subset of $\P(\Hom(E,F)^\vee)$ which contains all the maps $\vp: E \to F$ such that $\vp \vert_c =0$ for at least one point $c \in C$. 
Then $T$ is a closed subset and as a set 
$$T  = \bigcup_{c \in C} \P(\Hom(E,F\otimes \m_c)^\vee)$$
where $\m_c$ denotes the sheaf of ideals of the point $c$.
As $E^\vee \otimes F$ is globally generated and $H^1(E^\vee \otimes F) = 0$,
it follows that $H^1(E^\vee \otimes F\otimes \m_c) = 0$ for every closed point $c \in C$.
Using this, it follows that 
\begin{equation}
    \dim(\P(\Hom(E,F)^\vee)) - \dim T \geqslant rk-1 \geqslant r-k.
\end{equation}
Let $V$ denote the open set $\P(\Hom(E,F)^\vee) \setminus T$.
One can define a map $\nu : C \times V \to \P(E \otimes F^\vee)$ 
which sends a closed point $(c,[\vp])$ of $C \times V$ to the point $(c,\vp\vert_c)$ in $\P(E \otimes F^\vee)$.
The map $\nu$ satisfies the following diagram
$$
\begin{tikzcd}
    C \times V \arrow{r}{\nu} \arrow{rd}{} & \P(E \otimes F^\vee) \arrow{d}{\eta}\\
    & C 
\end{tikzcd}
$$
where $\eta$ is the structure map.
For any $c \in C$, the restriction 
$\nu_c : V \to \P(E\vert_c \otimes F^\vee\vert_c)$ is flat 
(for details, see proof of \cite[Lemma 7.12]{gs24}).
Now using \cite[\href{https://stacks.math.columbia.edu/tag/039E}{Tag 039E}]{Stk}, it follows that $\nu$ is flat.
Consider the canonical map on $\P(E \otimes F^\vee)$,
$$\eta^*E \longrightarrow \eta^*F \otimes \O_{\P(E \otimes F^\vee)}(1)$$
and let $Z$ be the support of the cokernel of the map above.
Then the set $Z \cap \eta^{-1}(c)$ is the locus of the non-surjective maps in $\P(E\vert_c \otimes F\vert_c^\vee))$.
Using \cite[Chapter II, section 2, page 67]{ACGH1}, we see that the codimension of 
$Z \cap \eta^{-1}(c) $ in $\P(E\vert_{c} \otimes F\vert_{c}^\vee))$ is $r-k+1$.
It follows that the codimension of $Z$ in $\P(E \otimes F^\vee)$ is $r-k+1$.
Since $\nu$ has constant fiber dimension, the codimension of $\nu^{-1}(Z)$ in $C \times V$ is $r-k+1$.
Let $p_V:C \times V \to V$ be the projection.
Hence the codimension of $p_V(\nu^{-1}(Z))$ in $V$ is $r-k$.
So the locus of the points in $\P(\Hom(E,F)^\vee)$ corresponding 
to non-surjective maps is the set $T \cup p_V(\nu^{-1}(Z))$, which has codimension at least $r-k \geqslant 1$.
This proves the lemma.
\end{proof}

\begin{lemma}\label{lemma irreducibility stable quotients}
There exists a number $d_2:= d_2(E,k)$ such that if $d \geqslant d_2$ then the locus of all stable locally free quotients in $Q_{k,d}(E)$ is irreducible.
\end{lemma}
\begin{proof}
    Using \cite[Lemma 5.2']{New78} we get the number $d'_1=k(2g-1+2k)$ such that any stable locally free sheaf of rank $k$ and degree $d \geqslant d'_1$ is globally generated.
    Define $$d_2(E,k) := \max\{d_1,d'_1,1\}\,.$$
    Assume $d \geqslant d_2$.
    Following the proof of the irreducibility of the moduli space of stable locally free sheaves 
    (see \cite[Theorem 8.5.2]{LP97}), 
    one can construct a space $R$ which parametrizes pairs $(L,e_L)$ where $L$ is a line bundle of degree $d$ on $C$ and $e_L$ is an extension of the form
    \begin{equation}\label{F contains trivial}
        e_L : \quad  0 \to \O_C^{k-1} \to F \to L \to 0\,.
    \end{equation}
    There is a universal sheaf $\F$ on $C \times R$.
    Let $R^s \subset R$ denote the open subset containing all the points such that $F$ is a stable locally free sheaf. 
    Let $M^s_C(k,d)$ denote the moduli space of stable locally free sheaves on $C$ of rank $k$ and degree $d$.
    The restriction of the universal sheaf $\F$ on $C \times R^s$ induces a morphism $\Phi : R^s \to M^s_C(k,d)$ which sends a point $(L,e_L:0 \to \O_C^{k-1} \to F \to L \to 0)$ to the sheaf $F$.
By the choice of $d$, the space $R$ is irreducible and the map $\Phi$ is surjective.    
In other words, the set $\{\F_r : r\in R^s \}$ is the set of all the stable locally free sheaves on $C$ of rank $k$ and degree $d$.

    Let $p_C : C \times R^s \to C$ and $p_R : C \times R^s \to R^s$ denote the projections.
    The sheaf ${p_R}_*(p_C^*E^\vee \otimes \F)$ is locally free on $R^s$ using Lemma \ref{lemma gg and H1 vanishing of EdualF} and Guauert's theorem.  
    The projectivization $\P:= \P(({p_R}_*(p_C^*E^\vee \otimes \F))^\vee)$ parametrizes pairs $(r,[q:E \to \F_r])$ where $r \in R^s$ and $q$ is a non-zero map.
    The subset $U \subset \P$ containing all the points such that $q$ is surjective, is open in $\P$.
Using the tautological bundle over $\P$, one gets a map of schemes $g :U \to Q_{k,d}(E)$ which sends the point $(r, [E \to \F_r])$  to the point $[E \to \F_r]$.
The image of $g$ is precisely the set of all stable locally free quotients of $E$ of rank $k$ and degree $d$.
As $R^s$ is irreducible, so is $\P$ and $U$.
Hence, we conclude that the set parametrizing the stable locally free quotients in $Q_{k,d}(E)$ is also irreducible.
\end{proof}

\begin{theorem}\label{theorem stable component}
If $d \geqslant d_2(E,k)$ then there is a unique component of $Q_{k,d}(E)$ whose general point corresponds to a quotient $[E \to F]$ 
such that $F$ is stable locally free. 
This component is generically smooth of dimension $dr-ke+k(r-k)(1-g)$.
\end{theorem}
\begin{proof}
    Fix a stable bundle $F$ of rank $k$ and degree $d \geqslant d_2$.
    Using Lemma \ref{lemma gg and H1 vanishing of EdualF}, we conclude that $\HOM(E,F)$ is generated by global sections and $H^1(E^\vee \otimes F) = 0$.
    Lemma \ref{lemma generic surjection} says that a general morphism $E \to F$ is surjective.
    Such a quotient $\vp : E \to F$ corresponds to a point in the Quot scheme $Q_{k,d}(E)$. 
    Let $\mc W$ be an irreducible component of $Q_{k,d}(E)$ which contains the quotient $[\vp:E \to F]$. 
    Let $S_F$ denote the kernel of the quotient.
    By Lemma \ref{lemma H1 vanishing of S_FdualF} we have $\h^1(S_F^\vee \otimes F)=0$ and
    $h^0(S_F^\vee \otimes F) = dr-ke+k(r-k)(1-g)$.
    This implies that the point $[\vp: E \to F]$ is a smooth point of $\mc W$
    and $\dim_{[\vp]}\mc W = dr-ke+k(r-k)(1-g)$.
    As the subset of stable locally free quotients is open in $Q_{k,d}(E)$, 
    a general point of $\mc W$ corresponds to stable locally free quotient. 
    This proves the existence of components whose general points correspond to stable locally free quotients and that any such component is generically smooth and of dimension $dr-ke+k(r-k)(1-g)$.
    The uniqueness of such component follows from 
    the irreducibility of the locus of all stable locally free quotients in $Q_{k,d}(E)$, which is proved in Lemma
    \ref{lemma irreducibility stable quotients}.
    \end{proof}

\begin{remark}\label{remark component P}
    Let $\mc P$ denote the closure of the locus of stable locally free quotients in $Q_{k,d}(E)$. 
    If $d \geqslant d_2$, by Theorem \ref{theorem stable component},
    it follows that $\mc P$ is an irreducible component of $Q_{k,d}(E)$.
    Moreover using Lemma \ref{lemma irreducibility stable quotients}, we see that, if $\mc W \neq \mc P$ is another component of $Q_{k,d}(E)$, then $\mc W$ cannot contain any stable locally free quotient.
\end{remark}

\section{Irreducible Component containing Torsion-free quotients}\label{section component of torsion-free quotients}

Let us first recall Serre duality for nodal curves.
As $C$ is a projective scheme over $\C$, by \cite[Proposition 7.5, Chapter III]{Ha}, $C$ has a dualizing sheaf $\omega^{\circ}_C$.
As the nodal curve $C$ is a local complete intersection, 
by \cite[Theorem 7.11, Chapter III]{Ha}, $\omega^{\circ}_C$ is locally free.
Moreover, by \cite[Theorem 7.6, Chapter III]{Ha}, we have an isomorphism
\begin{equation}\label{equation Serre duality}
 \Ext^i(G,\omega^{\circ}_C) \xrightarrow{\sim} H^{1-i}(C,G) 
\end{equation}
for any coherent sheaf $G$ on $C$ and $i\geqslant 0$.

\begin{lemma}\label{lemma H1 vanishing}
    There exists a number $d_3(E,k)$ such that if $d \geqslant d_3(E,k)$ then we have the following.
    Let $\mc W$ be an irreducible component of $Q_{k,d}(E)$ whose general point corresponds to a torsion-free quotient.
    Then a general point $[E \to F]$ of $\mc W$ satisfies $H^1(E^\vee \otimes F)=0$.
\end{lemma}
\begin{proof}
    First let us define $d_3$.
    Let $\mu_0:= \mu_{\max}(E \otimes \omega^{\circ}_C)+1$.
    For any $d$, consider the set $Y_{\mu_0}$ as defined in Lemma \ref{lemma dimension Ymu},
    $$Y_{\mu_0} = \{ [E \to F] \in Q^{\tf}_{k,d}(E): \mu_{\min}(F) < \mu_0 \}\,.$$
    Using Lemma \ref{lemma dimension Ymu}, we get constants 
    $C_3(E,k,\mu_0)$ and $\alpha_3(E,k,\mu_0)$ such that
    if $d \geqslant \alpha_3$ then
    $$ \dim Y_{\mu_0} \leqslant d(r-1) + C_3\,.$$
    We define
    $$d_3 := \max\{\alpha_3, C_3 -(-ke+k(r-k)(1-g)-n(r-k)(2r+k))+1\}\,.$$
    Thus, if $d \geqslant d_3$ then we have 
    \begin{equation}\label{equation if d > alpha?}
    \dim Y_{\mu_0} < dr-ke+k(r-k)(1-g)-n(r-k)(2r+k)\,.    
    \end{equation}
    Note that $\mu_0$ depends only on $E$ and hence the number $d_3$ depends only on $E$ and $k$.
    
    Now we assume $d \geqslant d_3$.
    Let $\mc W$ be an irreducible component of $Q_{k,d}(E)$ whose general point corresponds to a torsion-free quotient.
    Using Proposition \ref{proposition dimension bound},
    we have $\dim \mc W \geqslant [dr-ke+k(r-k)(1-g)-n(r-k)(2r+k)]$.
    As $d \geqslant d_3$, \eqref{equation if d > alpha?} is satisfied and hence $\dim Y_{\mu_0} < \dim \mc W$.
    It follows that a general torsion-free quotient in $\mc W$ is not in $Y_{\mu_0}$, which translates to the following.
    Let $[E \to F]$ be a general point of $\mc W$. By assumption, $F$ is torsion-free. 
    As $\mu_{\min}(F) > \mu_{\max}(E \otimes \omega^\circ_C)$,
    by \cite[Lemma 1.3.3]{HL}, it follows that $\Hom(F,E \otimes \omega^{\circ}_C) = 0$.
    Using Serre duality \eqref{equation Serre duality}, we get $H^1(E^\vee \otimes F)=0$.
    This proves the lemma.
\end{proof}

\begin{theorem}\label{theorem locally free component}
    There is a number $d_4(E,k)$ such that if $d \geqslant d_4$, then there exists a unique component of $Q_{k,d}(E)$ whose general point corresponds to a locally free quotient.
\end{theorem}

\begin{proof}
Define 
$$d_4:= \max\{d_2,d_3\}\,.$$
Assume $d \geqslant d_4$.
As $d_4 \geqslant d_2$, by Remark \ref{remark component P}, 
$Q_{k,d}(E)$ has a component $\mc P$ which contains all stable locally free quotients.
This proves the existence of a component whose general point corresponds to a locally free quotient.
We need to show the uniqueness of such a component.

Let $\mc W$ be any component of $Q_{k,d}(E)$ whose general point corresponds to a locally free quotient of $E$.
We will show that $\mc W$  and $\mc P$ are the same component.
Let $[\varphi : E \to F] \in \mc W$ be a general locally free quotient in $\mc W$.
By Lemma \ref{lemma H1 vanishing}, we have $H^1(E^\vee \otimes F) = 0$.

By \cite[Lemma 2.4]{bh06} there exists an irreducible variety $T$ and a $T$-flat sheaf $\F$ on $C \times T$ such that 
$\F_{t}$ is a stable vector bundle of rank $k$ and degree $d$ on $C$ for a general point $t \in T$
and $\F_{t_0} = F$ for some $t_0 \in t$.
Using $\F$ we want to construct a family of quotients of $E$.
Let $p_C : C \times T \to C$ and $p_T : C \times T \to T$ denote the projections. 
Consider the $T$-flat sheaf $p_C^*E \otimes \F$ on $C \times T$.
We replace $T$ by the open subset containing points $t$ such that 
$H^1(E^\vee \otimes \F_t) = 0$.
Note that $T$ is non-empty as $t_0 \in T$.
Using Grauert's theorem, it follows that ${p_T}_*(p_C^*E^\vee \otimes \F)$ is a locally free sheaf on $T$.
So the projectivization $\P := \P(({p_T}_*(p_C^*E^\vee \otimes \F))^\vee)$
is irreducible.
Let $\eta : \P \to T$ denote the natural projection.
We have the tautological quotient of sheaves on $\P$,
\begin{equation}\label{expand H}
\eta^*({p_T}_*(p_C^*E^\vee \otimes \F))^\vee \to \O_{\P}(1) \to 0\,.
\end{equation}
Let $\wt \eta$ and $\rho_\P$ denote the maps as shown in the following Cartesian diagram 
$$ 
\begin{tikzcd}
    C \times \P \arrow{r}{\wt \eta} \arrow{d}{\rho_\P} & C \times T \arrow{d}{p_T}\\
    \P \arrow{r}{\eta} & T 
    \end{tikzcd}
$$
and let $\rho_C : C \times \P \longrightarrow C$ denote the projection.
From \eqref{expand H}, one gets a map of sheaves on $C \times \P$,
\begin{equation}
\theta : \rho_C^*E \to \rho_\P^*(\O_\P(1)) \otimes \widetilde\eta^*\F \,.
\end{equation}
Let $U$ be the open subset of $\P$ defined as follows 
$$ U = \{ u \in \P : \theta\vert_{C \times \{u\}} : E \to \F_{\eta(u)} \textrm{ is surjective }\}\,.$$
The quotient $\varphi : E \to F$, we started with, corresponds to a point $u_0 \in \P$ such that $\eta(u_0) = t_0$.
Then the restriction of the map $\theta$ on the fiber $C \times \{u_0\}$ is given by the quotient $\varphi : E \to F$ itself.
So $u_0$ is in $U$, which proves the non-emptiness of $U$.
Also note that $\P$ is irreducible and hence so is $U$.
As $\F$ is flat over $T$, $\wt \eta^*\F$ is flat over $\P$ and consequently 
$\rho_\P^*(\O_\P(1)) \otimes \widetilde\eta^*\F$ is flat over $\P$.
So the quotient of sheaves $\theta$ restricted on $C \times U$ gives a map
$$ g : U \to Q_{k,d}(E)\,.$$
A closed point of $U$ corresponds to a pair $(t,\varphi_t)$ where $t \in T$ and $\varphi_t : E \to \F_t$ is surjective.
The map $g$ sends the point $(t,\varphi_t: E \to \F_t)$ to the point $[\varphi_t : E \to \F_t]$ in $Q_{k,d}(E)$.
Since a general point of $T$ is a stable vector bundle and the projection $\eta : \P \to T$ is surjective, $\F_t$ is a stable vector bundle for a general point $(t,\varphi_t : E \to \F_t)$ in $\P$.
As $U$ is irreducible, we conclude that a general point of $U$ corresponds to a stable locally free quotient. 
Let the corresponding open set be denoted by $U'$.
Then $g(U') \subset \mc P$.
This implies that $g(U) \subset g(\overline{U'}) \subset \overline {\mc P}= \mc P$.
So the image of $U$ is contained in $\mc P$ and hence $g(u_0)=[\vp: E \to F] \in \mc P$.
So a general element of $\mc W$ is in $\mc P$ which implies that $\mc W=\mc P$.
This proves the theorem.
\end{proof}

\begin{theorem}\label{theorem torsion-free component}
    If $d \geqslant d_4$, then there exists a unique component of $Q_{k,d}(E)$ whose general point corresponds to a torsion-free quotient of $E$.
\end{theorem}
\begin{proof}
As $d \geqslant d_4$, the existence of such a component is already proved by Theorem \ref{theorem locally free component}.
Uniqueness of such a component can be proved by using the same argument as in the proof of the Theorem \ref{theorem locally free component} and using \cite[Lemma 2.3]{bh06} (in place of \cite[Lemma 2.4]{bh06}).
\end{proof}

\begin{corollary}\label{corollary component P contains torsion-free}
    Let $d \geqslant d_4$. Let $\mc W$ be an irreducible component of $Q_{k,d}(E)$ which contains a torsion-free quotient. Then $\mc W = \mc P$. 
\end{corollary}
\begin{proof}   
    Let $[E \to F] \in \mc W$ be a quotient such that $F$ is torsion-free.
    Let $V=\{ [E \to G] \in \mc W : G \textrm{ is torsion-free\,}\}$.
    By \cite[Proposition 2.3.1]{HL}, $V$ is an open subset of $\mc W$.
    As $V$ is non-empty, $\mc W$ is a component of $Q_{k,d}(E)$ whose general point corresponds to a torsion-free quotient.
    Note that $\mc P$ is also a component of $Q_{k,d}(E)$ whose general point corresponds to a torsion-free quotient.
    By Theorem \ref{theorem torsion-free component}, it follows that 
    $\mc W = \mc P$.
\end{proof}

\cite[Lemma 2.4]{bh06} essentially states that a locally free sheaf on $C$ can be deformed to a stable locally free sheaf (of the same rank and degree).
The reason that Theorems \ref{theorem torsion-free component} and \ref{theorem locally free component} cannot be combined is the lack of a result that replaces the locally free sheaf in \cite[Lemma 2.4]{bh06} with a torsion-free sheaf.
It is remarked in \cite{bh06} that the proof of \cite[Lemma 2.4]{bh06} fails if the locally free sheaf is replaced by a torsion-free sheaf.
However, now this can be done (see Proposition \ref{proposition deformation of torsion free sheaves to stable locally free}) as an application of Theorem \ref{theorem torsion-free component}.
We fix an ample line bundle $\O(1)$ on $C$.

\begin{lemma}\label{lemma surjection to any twist}
    Let $F$ be a torsion-free sheaf on $C$ of rank $k$, degree $d$ and type $(a_1,\dots,a_n)$. Let $a = \max\{a_1,\dots,a_n\}$.
    Then there exists $\ell_0$ such that
    for any $\ell \geqslant \ell_0$ there is a surjection
    $\O_C^{\oplus 2k+a} \to F(\ell) \to 0$.
\end{lemma}
\begin{proof}
We take $\ell_0$ to be an integer such that if $\ell\geqslant \ell_0$ then $F(\ell)$ is generated by global sections.
Let us fix $\ell \geqslant \ell_0$.
Let $c \in C$ be a non-singular point.
So the natural map
 $$ \vp : H^0(F(\ell)) \otimes \O_C \longrightarrow F(\ell)$$ is surjective.
The vector space $(F(\ell)) \otimes \kappa(c)$ is generated by $k$ elements.
So there is a subspace $V_1 \subset  H^0(F(\ell))$ of dimension $k$ such that the natural map
$\vp_{V_1} :V_1 \otimes \O_C \to F(\ell)$
is surjective at the point $c$.
It follows that the map $\vp_{V_1}$ is surjective on an open set $U \subset C$ containing $c$.
The complement of $U$ in $C$ consists of finitely many points, say $c_1,\dots,c_s$ (including the nodes).
For each $c_i$,
the vector space $F(\ell)\vert_{c_i}$ has rank at most $k+a$.
The set of all subspaces $W \subset H^0(F(\ell))$ of rank $k+a$ such that the natural map $\vp_W : W \otimes O_C \to F(\ell)$ is surjective at $c_i$ is a non-empty open subset in the Grassmannian
${\rm Grass}( H^0(F(\ell)),k+a)$.
Hence we can find a subspace $V_2 \subset H^0(F(\ell))$ of rank $k+a$ such that the natural map $\vp_{V_2} : V_2 \otimes O_C \to F(\ell)$ is surjective at every $c_i$.
Let $V := V_1 + V_2$. 
Then the natural map 
$$ \vp_V : V \otimes O_C \longrightarrow F(\ell)$$
is surjective and rank of $V$ is at most $2k+a$.
The lemma follows.
\end{proof}

Combining Theorem \ref{theorem torsion-free component} and 
Lemma \ref{lemma surjection to any twist}, 
we get the Proposition \ref{proposition deformation of torsion free sheaves to stable locally free}.

\begin{proposition}\label{proposition deformation of torsion free sheaves to stable locally free}
    Let $F$ be a torsion-free sheaf on $C$ of rank $k$ and degree $d$.
    Then there exists a family $\F$ of torsion-free sheaves on $C$ of rank $k$ and degree $d$,
    parametrized by an integral scheme $T$
    such that for a general point $t$ in $T$, the sheaf $\F_t$ is stable locally free
    and $\F_{t_0} = F$ for some $t_0 \in T$.
\end{proposition}
\begin{proof}
Let the type of $F$ be $(a_1,\dots,a_n)$ and $a = \max\{a_1,\dots,a_n\}$.
For any degree $d'$, consider the Quot scheme $Q_{k,d'}(\O_C^{\oplus 2k+a})$.
Using Theorem \ref{theorem torsion-free component}, 
we get a constant $d'_4 = d_4(\O_C^{\oplus 2k+a},k)$, 
such that if $d' \geqslant d'_4$ then the torsion-free locus 
$Q^{\tf}_{k,d'}(\O_C^{\oplus 2k+a})$ is irreducible.  
Using Lemma \ref{lemma surjection to any twist},
we have $\ell_0$ such that if $\ell \geqslant \ell_0$ then there is a surjection 
$q: \O_C^{\oplus 2k+a} \to F(\ell) \to 0$.
Let us choose $\ell \geqslant \ell_0$ such that $\deg(F(\ell))=d+\ell k \geqslant d'_4$.
Then the surjection $q:\O_C^{\oplus 2k+a} \to F(\ell)$
corresponds to a closed point in the Quot scheme $Q_{k,d+\ell k}(\O_C^{\oplus 2k+a})$.
Moreover $F$ is torsion-free, so $[q]$ is a point of $Q^{\tf}_{k,d+\ell k}(\O_C^{\oplus 2k+a})$.
We take $T$ to be this torsion-free locus $Q^{\tf}_{k,d+\ell k}(\O_C^{\oplus 2k+a})$ with the reduced induced scheme structure.
As $d+\ell k \geqslant d'_4$,  $T$ is irreducible. 
Let $\F'$ denote the universal quotient sheaf on $C \times T$.
Then $\F'_q = F(\ell)$.
By Theorem \ref{theorem torsion-free component}, it follows that $\F'_t$ is stable locally free for a general point of $t \in T$.
Taking $\F := \F'(-\ell)$ the proposition follows.
\end{proof}

\section{Irreducibility of the Quot scheme}\label{section Irreducibility of the Quot scheme}
Let $E$ be a vector bundle on $C$ of rank $r$ and degree $e$ 
and let $k$ be an integer such that $0<k<r$.
In the previous section, we saw that for large $d$, there is an irreducible component of $Q_{k,d}(E)$ which contains the locus of torsion-free quotients.
In this section, we will show that for large $d$ the Quot scheme is irreducible.
First, we will prove the irreducibility for the case when $k>\frac{r}{2}$.
Then we will use this result to prove the irreducibility for all $k$.

Let $Z$ denote the complement of the torsion-free locus in the Quot scheme $Q_{k,d}(E)$. For a sheaf $F$, $\Tor(F)$ will denote the torsion subsheaf of $F$.
For $\delta >0$, let us define further subsets of $Z$,
\begin{align*}
    Z_\delta &:= \{[E \to F] \in Z: \quad \len(\Tor(F))=\delta\} \, ,\\
    Z^{\rm g} &:= \{ [E \to F] \in Z: \quad \text{the stalks $F_{x_1},\dots,F_{x_n}$ are free }\} \, , \\
    Z^{\rm g}_\delta &:= Z^{\rm g} \cap Z_\delta \,.
\end{align*}
\noindent
We emphasize that, for a point $[E \to F] \in Z^{\rm g}$, $\Tor(F)$ is supported away from the nodes and the stalks at the nodes $F_{x_1},\dots,F_{x_n}$ are free.
We endow all these subsets with the reduced induced subscheme structure.

\begin{lemma}\label{lemma dimension Zdelta}
    Assume $\frac{r}{2}<k<r$. For any $d$ and $\delta>0$, consider the subsets $Z_\delta$ and $Z^{\rm g}_\delta$ in $Q_{k,d}(E)$.
    Then we have
    \begin{align*}
        \dim Z_\delta & \leqslant \dim(Q_{k,d-\delta}^{\tf}(E)) +(r-1)\delta \, ,\\
        \dim Z^{\rm g}_\delta &\leqslant \dim(Q_{k,d-\delta}^{\tf}(E)) +(r-k)\delta \,.
    \end{align*}
    Moreover $Z_\delta$ is non-empty only when $d-\delta \geqslant m(E,k)$.
\end{lemma}
\begin{proof}
    Let $[\vp : E \to F]$ be a closed point in $Z_\delta$ and $\tau$ be the torsion subsheaf of $F$ which has length $\delta$.
    Then we have a torsion-free quotient $\vp' : E \to F'$,
    and a torsion quotient $\psi : S_{F'} \to \tau$,
    which fit into the following commutative diagram whose rows and columns are exact
    \begin{equation}\label{diagram torsion}
    \begin{tikzcd}
    & & 0 \arrow{d}{} & 0 \arrow{d}{}\\
    0 \arrow{r}{} & S_F \arrow{r}{} \ar[d,-,double equal sign distance,double] & S_{F'} \arrow{d}{} \arrow{r}{\psi} & \tau \arrow{d}{} \arrow{r}{} & 0 \\
    0 \arrow{r}{} & S_F \arrow{r}{} & E \arrow{d}{\vp'} \arrow{r}{\vp} & F  \arrow{r}{} \arrow{d}{} & 0 \\
    & & F' \ar[r,-,double equal sign distance,double]  \arrow{d}{} & F' \arrow{d}{} \\
    & & 0  & 0 \,. \\
    \end{tikzcd}
    \end{equation}
    On the other hand, given the torsion-free quotient 
    $\vp' : E \to F'$ and the torsion quotient $\psi : S_{F'} \to \tau$ of length $\delta$, one can get back the quotient $\vp: E \to F$ 
    using the following push-out diagram
    $$ 
    \begin{tikzcd}
        0 \ar{r}{} & S_{F'} \ar{r}{} \ar{d}{\psi} & E \ar{r}{\vp'} \ar{d}{\vp} & F' \ar[d,-,double equal sign distance,double] \ar{r}{} & 0 \phantom{\, .}\\
        0 \ar{r}{} & \tau \ar{r}{} & F \ar{r}{} & F' \ar{r}{} & 0 \,.
    \end{tikzcd}
    $$
    This correspondence gives us 
    \begin{align*}
    \dim Z_\delta \leqslant & \dim(Q_{k,d-\delta}^{\tf}(E)) + \max_{q' \in Q_{k,d-\delta}^{\tf}(E)} \left\{\dim(Q_{0,\delta}(S_{F'}) \right\}\,.
    \end{align*}
The above argument can be made precise by constructing a map from a relative Quot scheme onto the locus $Z_\delta$,
as in the proof of Lemma \ref{lemma dimension wt Z}.
    
    If $S_{F'}$ is of type $(a_1,\dots,a_n)$ and $a = \max\{a_1,\dots,a_n\}$ then $\dim(Q_{0,\delta}(S_{F'})) \leqslant(r-k+a)\delta$.
    The rank of $S_{F'}$ is $r-k$.
    As $k>\frac{r}{2}$, it follows that
    $r-k < \frac{r}{2}$. 
    As the type of $S_{F'}$ at any node can be at most the rank of $S_{F'}$,
    we get that $a \leqslant r-k < \frac{r}{2}<k$.
    Thus $a \leqslant k-1$, from which we have
    $\dim(Q_{0,\delta}(S_{F'})) \leqslant (r-1)\delta$. 
    Hence we have 
    $$\dim Z_\delta \leqslant \dim(Q_{k,d-\delta}^{\tf}(E)) + (r-1)\delta \,.$$
    Now for points in $Z^{\rm g}_\delta$, the sheaf $S_{F'}$ is locally free of rank $r-k$ and hence
    $\dim(Q_{0,\delta}(S_{F'})) = (r-k)\delta$. 
    So we have
    \begin{align*}
    \dim Z^{\rm g}_\delta \leqslant & \dim(Q_{k,d-\delta}^{\tf}(E)) + (r-k)\delta \,.
    \end{align*}

    Observe that $Z_\delta$ is non-empty only when the Quot scheme $Q_{k,d-\delta}^{\tf}(E)$ is non-empty.
    Recall that $m(E,k)$ is the minimal degree of a quotient of $E$ of rank $k$.
    So we must have $d-\delta \geqslant m(E,k)$.
    This proves the lemma.
\end{proof}

\begin{lemma}\label{lemma no general point of Z0delta}
    There exists a number $d_5(E,k)$ such that if $d \geqslant d_5$ then there is no component of $Q_{k,d}(E)$ whose general point is in $Z^{\rm g}$.
\end{lemma}
\begin{proof}
    First we define $d_5$. 
    Recall
    that $m(E,k)$ is the minimal degree of a quotient of $E$ of rank $k$. 
    Also recall that by Theorem \ref{theorem torsion-free component},
    if $d \geqslant d_4(E,k)$ then the torsion-free locus $Q^{\tf}_{k,d}(E)$ is irreducible of dimension $dr-ke+k(r-k)(1-g)$.
    Let us define
    $$ M = \max_{m(E,k) \leqslant d' < d_4}\left\{\dim Q^{\tf}_{k,d'}(E) - [d'r-ke+k(r-k)(1-g)] +d'k\right\}$$
    and $$d_5 := \max\{d_4, \frac{M}{k}+1\}\,.$$

    Now we assume $d \geqslant d_5$. 
    If possible, let $\mc W$ be a component of $Q_{k,d}(E)$ whose general point corresponds to a quotient $[\vp: E \to F]$ in $Z^{\rm g}_{\delta}$, where $\delta>0$. 
    Using Lemma \ref{lemma dimension Zdelta}, we have $d-\delta \geqslant m(E,k)$.
    As the stalks $F_{x_1},\dots,F_{x_n}$ are free,
    the kernel $S_F$ of $\vp$ is locally free.
    The dimension bound in \eqref{dimension bound} gives 
    \begin{align*}
        \dim_{[\vp]}Q_{k,d}(E) &\geqslant \hom(S_F,F) - \ext^1(S_F,F) \\ 
        & = \chi(S_F^\vee \otimes F) \\
        & = dr-ke+k(r-k)(1-g)\,.
    \end{align*}
    So we have 
    \begin{equation}
     \dim \mc W \geqslant  dr-ke+k(r-k)(1-g)\,.
    \end{equation}
    Also from Lemma \ref{lemma dimension Zdelta}, we have
    \begin{align*}
    \dim Z^{\rm g}_\delta \leqslant \dim(Q_{k,d-\delta}^{\tf}(E)) + (r-k)\delta \,.
    \end{align*}
    So
    \begin{align*}
    \dim Z^{\rm g}_\delta - \dim \mc W & \leqslant \dim Q^{\tf}_{k,d-\delta}(E) + (r-k)\delta - [dr-ke+k(r-k)(1-g)]  \\
    & = \dim Q^{\tf}_{k,d-\delta}(E) - [(d-\delta)r-ke+k(r-k)(1-g)] -\delta k \,.
    \end{align*}
    Observe that if $d-\delta \geqslant d_4$, then by Theorem \ref{theorem torsion-free component},
    $Q^{\tf}_{k,d-\delta}(E)$ has dimension $[(d-\delta)r-ke+k(r-k)(1-g)]$.
    So the right hand side is negative. 
    Otherwise we have $m(E,k) \leqslant d-\delta < d_4$. As $d > M/k$, in this case we have
    $$dk > M \geqslant \dim Q^{\tf}_{k,d-\delta}(E) - [(d-\delta)r-ke+k(r-k)(1-g)] + (d-\delta)k$$
    which implies that 
    $$ \dim Q^{\tf}_{k,d-\delta}(E) - [(d-\delta)r-ke+k(r-k)(1-g)] - \delta k < 0\,.$$
    So the difference $\dim Z^{\rm g}_\delta - \dim \mc W$ is negative.
    This shows that the points of $Z^{\rm g}_\delta$ cannot be general in $\mc W$.
\end{proof}

\begin{proposition}\label{Zg in P}
    For $d \geqslant d_5(E,k)$, the locus $Z^{\rm g}$ in $Q_{k,d}(E)$ is contained in the closure of the torsion-free locus $Q^{\tf}_{k,d}(E)$.
\end{proposition}
\begin{proof}
    Recall the irreducible component $\mc P$ defined in Remark \ref{remark component P}, which contains the locus of torsion-free quotients in $Q_{k,d}(E)$.
    We need to show that $Z^{\rm g}$ is contained in $\mc P$.
    Let, if possible, $[\vp : E \to F]$ be a quotient in $Z^{\rm g}$ which is not in the component $\mc P$.
    Let $\mc W$ be a component containing the point $[\vp : E \to F]$.
    Consider the open subset $V = \{[q: E \to G] \in \mc W :$ the stalks $G_{x_1},\dots,G_{x_n}$ are free\}.
    Note that $[\vp: E \to F] \in V$, which shows that $V$ is non-empty.
    For any quotient $[q: E \to G] \in V$, either $G$ is torsion-free or $[q: E \to G] \in Z^{\rm g}$.
    If $V$ contains a torsion-free quotient then we must have $\mc W = \mc P$, 
    by Corollary \ref{corollary component P contains torsion-free}.
    This is a contradiction.
    So $V$ cannot contain any torsion-free quotient.
    In other words, any point of $V$ is contained in $Z^{\rm g}$.
    But this is a contradiction as the points of
    $Z^{\rm g}$ cannot be general in any component by Lemma \ref{lemma no general point of Z0delta}.
    We conclude that the point $[\vp:E \to F]$ is in $\mc P$ and hence $Z^{\rm g}$ is contained in $\mc P$. 
\end{proof}
For ease of notation, let us define the number
\begin{equation}
     \delta_0 := n(r-k)(2r+k)\,.
\end{equation}
Recall Proposition \ref{proposition dimension bound}, which says that any irreducible component of $Q_{k,d}(E)$ has dimension at least $[dr-ke+k(r-k)(1-g)] - \delta_0$.

\begin{lemma}\label{lemma bound on torsion for general point}
    Assume $\frac{r}{2}<k<r$. There exists a number $d_6(E,k)$ such that if $d \geqslant d_6$ then we have the following.
    If $\mc W$ is any component of $Q_{k,d}(E)$ and $[E \to F]$ is a general point of $\mc W$ then length$(\Tor(F)) \leqslant \delta_0$.
\end{lemma}
\begin{proof}
We will use the same technique as in Lemma \ref{lemma no general point of Z0delta}. 
First let us define $d_6(E,k)$.
Define
    $$ M = \max_{m(E,k) \leqslant d' < d_4}
    \left\{\dim Q^{\tf}_{k,d'}(E) - (d'r-ke+k(r-k)(1-g)) + \delta_0 + d'\right\}$$
    and $$d_6 := \max\{d_4, M+1\}\,.$$
    
    Now we assume $d \geqslant d_6$ and $\delta > \delta_0$.
    If possible, let $\mc W$ be a component of $Q_{k,d}(E)$ whose general point corresponds to a quotient $[\vp: E \to F]$ in $Z_{\delta}$. 
    The dimension bound from Proposition \ref{proposition dimension bound} gives 
    \begin{equation}
     \dim \mc W \geqslant  dr-ke+k(r-k)(1-g) - \delta_0\,.
    \end{equation}
    From Lemma \ref{lemma dimension Zdelta}, we have
    \begin{align*}
    \dim Z_\delta \leqslant \dim(Q_{k,d-\delta}^{\tf}(E)) +(r-1)\delta \,.
    \end{align*}
    So
    \begin{align*}
    \dim Z_\delta - \dim \mc W & \leqslant \dim Q^{\tf}_{k,d-\delta}(E) + (r-1)\delta - [dr-ke+k(r-k)(1-g)] +\delta_0  \\
    & = \dim Q^{\tf}_{k,d-\delta}(E) - [(d-\delta)r-ke+k(r-k)(1-g)] -\delta +\delta_0 \,.
    \end{align*}
    Observe that, if $d-\delta \geqslant d_4$ then by Theorem \ref{theorem torsion-free component},
    $Q^{\tf}_{k,d-\delta}(E)$ has dimension $[(d-\delta)r-ke+k(r-k)(1-g)]$.
    In this case, the right-hand side is negative as $\delta > \delta_0$. 
    Otherwise we have $m(E,k) \leqslant d-\delta < d_4$. As $d > M$, in this case we have
    $$d > M \geqslant \dim Q^{\tf}_{k,d-\delta}(E) - [(d-\delta)r-ke+k(r-k)(1-g)] + \delta_0 + (d-\delta)$$
    which implies that 
    $$ \dim Q^{\tf}_{k,d-\delta}(E) - [(d-\delta)r-ke+k(r-k)(1-g)] - \delta +\delta_0 < 0\,.$$
    So the difference $\dim Z_\delta - \dim \mc W$ is negative.
    This shows that the points of $Z_\delta$ cannot be general in $\mc W$.
\end{proof}

\begin{proposition}\label{proposition Ext1(S,E) vanishinig}
    Assume $\frac{r}{2}<k<r$. There exists a number $d_7(E,k)$ such that if $d \geqslant d_7$
    then the following holds.
    Let $\mc W$ be an irreducible component of $Q_{k,d}(E)$ whose general point corresponds to a quotient in $Z_\delta$, with $\delta >0$.
    Then a general point $[q:E \to F]$ of $\mc W$ satisfies $\Ext^1(\ker q,E)=0$.
\end{proposition}
\begin{proof}
First we define $d_7$.
Let $\mu_0:= \mu_{\min}(E \otimes (\omega^\circ_C)^\vee)+1$.
    Using Lemma \ref{lemma dimension Xmu}, we have constants $C_4(E,k,\mu_0)$ and 
    $\alpha_4(E,k,\mu_0)$. 
    Define 
    $$d_7 := \max\{d_6, \alpha_4+\delta_0, 2\delta_0+C_4-[-ke+k(r-k)(1-g)]+1\}\,.$$
Consider the subset $X_{\mu_0}$ in $Q_{k,d-\delta}(E)$ with the reduced induced subscheme structure,
   $$X_{\mu_0} = \{ [q':E \to F'] \in Q_{k,d-\delta}(E): \mu_{\max}(\ker q') > \mu_0 \}\,.$$
By our choice of $d_7$ and Lemma \ref{lemma dimension Xmu},
if $d \geqslant d_7$ then we have
\begin{equation}\label{equation dimension Xmu}
       \dim X_{\mu_0} < dr-ke+k(r-k)(1-g)-(r+1)\delta_0\,.
\end{equation}

    Assume $d \geqslant d_7$.
    Let $\mc W$ be an irreducible component of $Q_{k,d}(E)$ whose general point is in $Z_\delta$ for some $\delta >0$.
As $d \geqslant d_6$, using Lemma \ref{lemma bound on torsion for general point} we have $\delta \leqslant \delta_0$.
Let $D$ denote the torsion-free locus of the Quot scheme $Q_{k,d-\delta}(E)$ i.e.
    $$ D := Q^{\tf}_{k,d-\delta}(E)\,.$$
    Let $\S'$ denote the universal kernel sheaf on $C \times D$.
    Consider the relative Quot scheme 
    $$H := \Quot_{C \times D/D}(\S',0,\delta)\,.$$
    A closed point of $H$ corresponds to the pair of quotients 
    $(q':E \to F', \lambda: \ker q' \to \tau)$.
    Given the point $(q',\lambda)$ in $H$, one can construct a quotient $q:E \to F$ using the following push-out diagram
    $$ 
    \begin{tikzcd}
        0 \ar{r}{} & \ker q' \ar{r}{} \ar{d}{\lambda} & E \ar{r}{q'} \ar{d}{q} & F' \ar[d,-,double equal sign distance,double] \ar{r}{} & 0 \phantom{\,.}\\
        0 \ar{r}{} & \tau \ar{r}{} & F \ar{r}{} & F' \ar{r}{} & 0 \,.
    \end{tikzcd}
    $$
    It is clear that the quotient $[q:E \to F]$ corresponds to a point in the Quot scheme $Q_{k,d}(E)$.
    We construct a map of schemes
    $$g : H \to Q_{k,d}(E)$$ 
    which sends the point $(q',\lambda)$ to the point $[q:E \to F]$.
    Let $\sigma :C \times D \to C$ 
    and $\xi: (C \times D) \times_D H \to C \times D$ denote the projections.
    Let $$0 \to \S' \to \sigma^*E \to \F' \to 0$$
    denote the universal sequence on $C \times D$ and 
    $$ \xi^*\S' \to \T \to 0$$
    denote the universal quotient on $(C \times D) \times_D H \cong C \times H$.
    Now one can construct a quotient $\xi^*\sigma^*E \to \B$ using the following push-out diagram
    $$ 
    \begin{tikzcd}
        0 \ar{r}{} & \xi^*\S' \ar{r}{} \ar{d}{} & \xi^*\sigma^*E \ar{r}{} \ar{d}{} & \xi^*\F' \ar[d,-,double equal sign distance,double] \ar{r}{} & 0 \phantom{\,.}\\
        0 \ar{r}{} & \T \ar{r}{} & \B \ar{r}{} & \xi^*\F' \ar{r}{} & 0 \,.
    \end{tikzcd}
    $$
    As $\F'$ is flat over $D$, $\xi^*\F'$ is flat over $H$.
    The sheaf $\Tau$, being the universal quotient, is flat over $H$.
    From the bottom short exact sequence, it follows that $\B$ is also flat over $H$.
    So the $H$-flat quotient $\xi^*\sigma^*E \to \B \to 0$ on $C \times H$ gives us the required map
    $$g: H \to Q_{k,d}(E)\,.$$
    From the pointwise description, it is clear that $H$ maps onto the subset $Z_\delta$. 
    Thus, there is an irreducible component, say $H'$, of $H$ which maps to $\mc W$
    and $g(H')$ is dense in $\mc W$.
    There is a structure map $\rho : H \to D$ which sends the closed point $(q',\lambda)$ to $[q']$.
    Let $D':= \rho(H')$.
    We have the following diagram where all the schemes are irreducible,
    $$ 
    \begin{tikzcd}
        H' \ar{r}{g} \ar{d}{\rho} & \mc W  \\
        D' \,.
    \end{tikzcd}
    $$
    Let $[q':E \to F']$ be any closed point in $D'$.
    The fiber of $\rho$ over the point $[q']$
    is contained in the Quot scheme $Q_{0,\delta}(\ker q')$.
    If the torsion-free sheaf $\ker q'$ has type $(a_1,\dots,a_n)$
    and $a=\max\{a_1,\dots,a_n\}$ then the Quot scheme $Q_{0,\delta}(\ker q')$ has dimension at most $(r-k+a)\delta$.
    By Lemma \ref{lemma type of kernel}\,(i), the type of $F'$ is also $(a_1,\dots,a_n)$ and so $a \leqslant k$.
    Consequently the fiber $\rho^{-1}([q'])$ has dimension at most $r\delta$.
    So we have 
    \begin{align*}
        \dim \mc W  \leqslant \dim H' 
         \leqslant \dim D' + r\delta 
         \leqslant \dim D' + r\delta_0 
    \end{align*}
    as $\delta \leqslant \delta_0$.
    In other words we have $\dim D' \geqslant \dim \mc W - r\delta_0$.
    Using the lower bound of $\dim \mc W$ from Proposition \ref{proposition dimension bound}, 
    we get that 
    \begin{equation}\label{equation dim D'}
        \dim D' \geqslant dr-ke+k(r-k)(1-g)-(r+1)\delta_0\,.
    \end{equation}
    As $d \geqslant d_7$, it follows that $\dim D' > \dim X_{\mu_0}$.
    So a general point of $D'$ is not in $X_{\mu_0}$. 
This implies that a general point $[q':E \to F']$ in $D'$ satisfies the inequality $\mu_{\max}(\ker q')\leqslant \mu_0$.
    Let $[q':E \to F'] \in D'$ be such a general point.
    Choose a point $(q',\lambda)$ in the fiber $\rho^{-1}([q'])$ and 
    let the image $g(q',\lambda)$ be $[q]$.
    It follows from the description of the map $g$ that $\ker q \subset \ker q'$.
    Thus $$\mu_{\max}(\ker q) \leqslant \mu_{\max}(\ker q')\leqslant \mu_0\,.$$
    This says that the point $[q]$ in $\mc W$, satisfies that $\mu_{\max}(\ker q) < \mu_{\min}(E \otimes (\omega^\circ_C)^\vee)$ and hence
    $\Hom(E \otimes (\omega^\circ_C)^\vee,\ker q)=0$.
    Using Serre duality \eqref{equation Serre duality} we get that 
    $ \Ext^1(\ker q, E)=0\,.$
    Using \cite[Theorem 1.4]{Lan83}, we conclude 
    the proof of the proposition.
\end{proof}

\begin{lemma}\label{lemma general cokernel free at node}
    There exists a number $d_8(E,k)$ such that for any stable locally free sheaf $S$ on $C$ of rank $r-k$ and degree $d \leqslant d_8$, the sheaf $S^\vee \otimes E$ is generated by global sections and $H^1(S^\vee \otimes E) =0$.
    For such a sheaf $S$ and a general homomorphism $\psi: S \to E$, the cokernel of $\psi$ is free at each node $x_i$.
\end{lemma}
\begin{proof}
    In Lemma \ref{lemma gg and H1 vanishing of EdualF}, replacing $E$ by $E^\vee$ and $k$ by $r-k$, we get the number $d'_1:= d_1(E^\vee, r-k)$
    such that for any stable locally free sheaf $G$ on $C$ of rank $r-k$ and degree $d \geqslant d'_1$, the sheaf $E \otimes G$ is globally generated and $H^1(E \otimes G)=0$.
    We take 
    $$ d_8:= -d'_1 \,.$$
    Let $S$ be any stable locally free sheaf on $C$ of rank $r-k$ and degree $d \leqslant d_8$.
    Then $S^\vee$ is stable locally free of rank $r-k$ and degree $-d\geqslant d'_1$.
    So $S^\vee \otimes E$ is generated by global sections and $H^1(S^\vee \otimes E)=0$.
    This proves the first assertion.
    
    For the second assertion, note that a general homomorphism of the vector spaces $S\vert_{x_i} \to E\vert_{x_i}$ is injective.
    As $S^\vee \otimes E$ is globally generated, it follows that for a general homomorphism $ \psi: S \to E$, the restriction 
    $\psi\vert_{x_i} : S\vert_{x_i} \to E\vert_{x_i}$ is injective and consequently the cokernel of $\psi$ is free at ${x_i}$ for all $i=1,2,\dots,n$.
\end{proof}

\begin{lemma}\label{lemma injective locus}
Let $Y$ be an irreducible Noetherian scheme and 
let $\mc G_1$ and $\mc G_2$ be coherent sheaves on $C \times Y$ which are flat over $Y$. 
Let $\phi:\mc G_1\to \mc G_2$ 
be a map of sheaves. 
The locus of points $y\in Y$, where $\phi_y$ is injective, 
is an open subset of $Y$. 
\end{lemma}
\begin{proof}
We will use the existence of flattening stratification, see \cite[\textsection 5.4.2]{FGA}. 
Let us assume that the set $y \in Y$ such that $\phi_y$ is injective
is nonempty, or else, there is nothing to prove. 
Let $\mc G$ denote the cokernel of $\phi$. 
Let $I$ denote the finite set consisting of the Hilbert polynomials of the sheaves $\{\mc G_y : y \in Y\}$. 
Let $f_1<f_2<\ldots<f_r$ be the elements of $I$ with the usual total
ordering on polynomials ($f<g$ iff $f(n)<g(n)$ for all $n\gg0$). 
Using \cite[Theorem 5.13]{FGA}, we get locally closed subschemes $Y_{f_i}$ 
where $Y_{f_i}$ consists of all the points $y \in Y$ for which the Hilbert polynomial of $\G_y$ is $f_i$.
Moreover the subset $Y_{f_1}$ is open.
We have an exact sequence 
$$0\to \ker(\phi_y) \to \mc G_{1,y}\xrightarrow{\phi_y} \mc G_{2,y}\to\mc G_y\to 0\,.$$
We denote the Hilbert polynomial of a sheaf $F$ on $C$ by ${\rm Hilb}(F)$.
Hence we have
$$ {\rm Hilb}(\mc G_y)={\rm Hilb}(\mc G_{2,y})-{\rm Hilb}(\mc G_{1,y})+{\rm Hilb}(\ker(\phi_y))\,.$$
Let $y_0$ be a point for which $\phi_{y_0}$ is injective. 
Then $${\rm Hilb}(\mc G_{y_0}) = {\rm Hilb}(\mc G_{2,y_0})-{\rm Hilb}(\mc G_{1,y_0}) \,.$$
Consequently, by flatness of $\G_1$ and $\G_2$, we have
\begin{align*}
    {\rm Hilb}(\mc G_y)&={\rm Hilb}(\mc G_{2,y})-{\rm Hilb}(\mc G_{1,y})+{\rm Hilb}(\ker(\phi_y))\\
    &={\rm Hilb}(\mc G_{2,y_0})-{\rm Hilb}(\mc G_{1,y_0})+{\rm Hilb}(\ker(\phi_y))\\
    & = {\rm Hilb}(\mc G_{y_0})+{\rm Hilb}(\ker(\phi_y))\,.
\end{align*}
Thus for any point $y \in Y$ we have 
${\rm Hilb}(\mc G_y)\geqslant {\rm Hilb}(\mc G_{y_0})$. 
Moreover, equality holds iff $\ker(\phi_y)=0$. 
This shows that ${\rm Hilb}(\mc G_{y_0})=f_1$
and that the set of points 
for which $\phi_y$ is injective is precisely the set $Y_{f_1}$. 
Hence the set of points $y$ for which $\phi_y$ is injective is open.
\end{proof}

\begin{theorem}\label{main theorem k>r/2}
Let $k$ be an integer such that $\frac{r}{2}<k<r$. Then there is a number $d_9(E,k)$ 
such that if $d \geqslant d_9$ then
the Quot scheme $Q_{k,d}(E)$ is irreducible.
\end{theorem}
\begin{proof}
    We define
    $$ d_9:= \max\{d_4, d_5, d_7, e-d_8\}\,.$$
    Assume $d \geqslant d_9$.
    As $d \geqslant d_4$, recall that, by Theorem \ref{theorem torsion-free component} or Corollary \ref{corollary component P contains torsion-free}, there is a unique irreducible component $\mc P$ of $Q_{k,d}(E)$ which contains the torsion-free locus $Q^{\tf}_{k,d}(E)$. 
    We want to show that there is no other component of the Quot scheme $Q_{k,d}(E)$.
    
    On the contrary, let us assume that $\mc W$ is an irreducible component of $Q_{k,d}(E)$ different from $\mc P$.
    So any closed point in $\mc W$ corresponds to a quotient with torsion.
    Let $[q:E \to F]$ be a general point in $\mc W$.
    Let $S_F$ denote the kernel of $q$.
    Then $S_F$ is a torsion-free sheaf on $C$ of rank $r-k$ and degree $e-d$.
    Using Proposition \ref{proposition deformation of torsion free sheaves to stable locally free}, we get a family $\S$ of torsion-free sheaves on $C$ of rank $r-k$ and degree $e-d$, parametrized by an integral scheme $T$ such that the sheaf $\S_t$ is stable locally free for a general point $t \in T$ and $\S_{t_0}=S_F$ for some $t_0 \in T$.
    Let $p_C : C \times T \to T$ and $p_T: C \times T \to T$ denote the projections.
    We consider the sheaf $\HOM(\S,p_C^*E)$ and let $\G$ be its push-forward on $T$,
    $$ \G := {p_T}_*(\HOM(\S,p_C^*E))\,.$$
    Proposition \ref{proposition Ext1(S,E) vanishinig} shows that we can assume $\Ext^1(S_F,E)=0$.
    We replace $T$ by the open subset containing points $t$ such that $\Ext^1(\S_t,E)=0$.
    Clearly $t_0 \in T$.
    Using \cite[Theorem 1.4]{Lan83}, we conclude that $G$ is a locally free sheaf on $T$.
    So the projectivization $\P:=\P(\G^\vee)$ is irreducible.
    Let $\eta : \P \to T$ be the structure map.
    We have the tautological quotient of sheaves on $\P$,
    $$ \eta^* \G^\vee \longrightarrow \O_{\P}(1) \longrightarrow 0\,,$$
    that is,
    \begin{equation}\label{tautological map in main theorem}
    \eta^* ({p_T}_*(\HOM(\S,p_C^*E)))^\vee \longrightarrow  \O_{\P}(1)    \longrightarrow 0\,. 
    \end{equation}
    Let $\rho_\P$ and $\wt \eta$ denote the maps as shown in the following Cartesian diagram and $\rho_C: C \times \P \to \P$ denote the projection
    $$ 
    \begin{tikzcd}
    C \times \P \arrow{r}{\wt \eta} \arrow{d}{\rho_\P} & C \times T \arrow{d}{p_T}\\
    \P \arrow{r}{\eta} & T \,.
    \end{tikzcd}
    $$
    Dualizing \eqref{tautological map in main theorem} and taking adjoint, we get a map of sheaves on $C \times \P$, 
    $$\rho_\P^*\O_\P(-1) \longrightarrow \wt\eta^*\HOM(\S,p_C^*E)\,.$$
    Tensoring with $\wt\eta^*\S$ 
    and composing with the natural map
    $\wt\eta^*\Hom(\S,p_C^*E) \otimes \wt\eta^*\S  \to \wt\eta^*p_C^*E$,
    we get a map of sheaves on $C \times \P$,
    $$\Psi: \wt\eta^*\S \otimes \rho_\P^*\O_\P(-1) \longrightarrow \wt\eta^*p_C^*E\,.$$
    A closed point $u \in \P$ corresponds to a pair $(t=\eta(u),\psi_u : \S_t \to E)$, where $\psi_u$ is a non-zero map.
    Restriction of $\Psi$ on $C \times \{u\}$ is the map $\psi_u$ itself.
    Let $U$ be the following subset of $\P$,
    $$ U := \{ u \in \P: \Psi\vert_{C \times \{u\}} : \S_{\eta(u)} \longrightarrow E \textrm{ is injective }\}\,.$$
    Then $U$ is open in $\P$ by Lemma \ref{lemma injective locus}.
    There is a point $u_0 \in \P$ which corresponds to the pair $(t_0, S_F \hookrightarrow E)$.
    This point is in $U$, which shows that $U$ is non-empty.
    Let $\B$ be the cokernel of $\Psi\vert_{C \times U}$.
    Then $\B$ is flat over $U$.
    The surjection of sheaves $\rho_C^*E \longrightarrow \B \longrightarrow 0$ on $C \times U$ gives a morphism of schemes
    $$ g : U \longrightarrow Q_{k,d}(E)\,.$$
    The map $g$ sends a closed point $u \in U$ to the point $[E \to \B_u]$ in the Quot scheme.
    Recall that a general member of the family $\S$ is a stable locally free sheaf on $C$.
    So for a general point $u \in U$,
    $\S_{\eta(u)}$ is a stable locally free sheaf on $C$ of rank $r-k$ and degree $e-d$.
    As $d \geqslant e-d_8$, $\deg(\S_{\eta(u)}) = e-d \leqslant d_8$.
    Using Lemma \ref{lemma general cokernel free at node}, we conclude that for a general morphism $\alpha : \S_{\eta(u)} \longrightarrow E$, the cokernel of $\alpha$ is free at each node $x_i \in C$. 
    In other words, for a general point $u \in U$, its image 
    $g(u) = [E \to \B_u]$ is contained in the locus $Z^{\rm g}$ of $Q_{k,d}(E)$.
    Lemma \ref{Zg in P} shows that the locus $Z^{\rm g}$ lies in the component $\mc P$. 
    So the image of a general point of $U$ is in $\mc P$.
    Hence the image of whole $U$ is in $\mc P$.
    In particular, the point we started with, that is, $[q:E \to F]=g(u_0)$ is in $\mc P$.
    As $[q:E \to F]$ is a general point in $\mc W$, this proves that $\mc W = \mc P$, which is a contradiction.
This proves that the Quot scheme $Q_{k,d}(E)$ is irreducible.
\end{proof}

\begin{theorem}\label{theorem main}
    Let $g$ and $n$ be integers such that $g>n>0$.
    Let $C$ be a reduced irreducible projective curve over $\C$ of arithmetic genus $g$ with nodes at $n$ points and smooth elsewhere.
    Let $E$ be a vector bundle on $C$ of rank $r$ and degree $e$ 
    and let $k$ be an integer such that $0<k<r$.
    Then there is a number $d_Q(E,k)$ such that if $d \geqslant d_Q(E,k)$ then the Quot scheme $Q_{k,d}(E)$ is irreducible.
    Moreover, $Q_{k,d}(E)$ is generically smooth and has dimension 
    $dr-ke+k(r-k)(1-g)$.
\end{theorem}
\begin{proof}
If $k>\frac{r}{2}$ then we are done using Theorem \ref{main theorem k>r/2}, taking $d_Q(E,k):= d_9(E,k)$.
Otherwise we have $0<k \leqslant \frac{r}{2}$.
In this case, we will use Theorem \ref{main theorem k>r/2} to prove irreducibility of the Quot scheme. 
    Let $t$ be the integer $t:= r-2k+1 > 0$
    so that we have $$ k+t > \frac{r+t}{2}\,.$$
    We define $r' := r+t$ and $k':= k+t$.
    So the new pair $r',k'$ satisfies 
    $$ r' > k' > \frac{r'}{2}\,.$$
    Let $E'$ denote the vector bundle $E \oplus \O_C^{\oplus t}$ which has rank $r'$ and degree $e$.
    Applying Theorem \ref{main theorem k>r/2} for the bundle $E'$,
    we get a number $d'_9:=d_9(E',k')$
    such that if $d \geqslant d'_9$ then the Quot scheme $Q_{k',d}(E')$ is irreducible.
    We will give a surjective map from an open subset of $Q_{k',d}(E')$ to the Quot scheme $Q_{k,d}(E)$.
    That will prove irreducibility of $Q_{k,d}(E)$ when $d \geqslant d'_9$.

    We define $$d_Q := \max\{d'_9,d_2(E,k)\}\,.$$ Assume $d \geqslant d_Q$.
    Let $Q'$ denote the reduced scheme structure on the irreducible Quot scheme $Q_{k',d}(E')$.
    So $Q'$ is integral.
    There is a universal short exact sequence of sheaves on $C \times Q'$,
    $$0 \to \S' \to p_C^*E' \to \F' \to 0 $$
    where $p_C : C \times Q' \to Q'$ is the projection.
    Let $\Psi$ denote the composition
    $$\S' \hookrightarrow p_C^*E' = p_C^*(E \oplus \O_C^{\oplus t}) \twoheadrightarrow p_C^*E\,.$$
    Let $U$ denote the open locus of all points $q' \in Q'$ for which the restriction map $\Psi\vert_{C \times \{q'\}}$ is injective.
    The sheaf $\coker(\Psi)\vert_{C \times U}$ is flat over $U$.
    So the quotient of sheaves $p_C^*E\vert_{C \times U} \to \coker(\Psi)\vert_{C \times U} \to 0 $ 
    on $C \times U$ defines a map of schemes
    $$g : U \longrightarrow Q_{k,d}(E)\,.$$
    Let us see the pointwise description of the map $g$.
    Let $[q': E' \to F'] \in Q_{k',d}(E')$ be a closed point in $U$.
    Let $S'$ denote the kernel of $q'$.
    As $[q']$ is in $U$, the composition 
    $S' \hookrightarrow E' \to E$ is injective.
    Let $B$ denote the cokernel of this inclusion $S \hookrightarrow E$.
    Then $B$ is a sheaf of rank $k$ and degree $d$.
    The map $g$ sends the point $[q': E' \to F']$ to the point $[E \to B]$.
    To see that the map $g$ is surjective on closed points, 
    let $[q:E \to F]$ be a closed point in the Quot scheme $Q_{k,d}(E)$.
    Using $q$, we get another quotient 
    $$\wt q:= q \oplus Id : E' = E \oplus \O_C^{\oplus t} \longrightarrow F \oplus \O_C^{\oplus t}\,.$$
    Clearly $F \oplus \O_C^{\oplus t}$ is a sheaf of rank $k'$ and degree $d$.
    So $\wt q$ corresponds to a closed point in $Q'$.
    Moreover, $\ker \wt q = \ker q \oplus 0 \subset E'$, 
    which implies that $[\wt q]$ is contained in the open set $U$
    and $g([\wt q])=[q]$.
    This shows that the map $g$ is surjective on closed points.
    As $Q'$ is irreducible, the Quot scheme $Q_{k,d}(E)$ is also irreducible. This proves the theorem.

    The second assertion follows from Theorem \ref{theorem stable component}, as $d \geqslant d_2(E,k)$.
\end{proof}

We have proved that for $d \geqslant d_Q(E,k)$, the Quot scheme $Q_{k,d}(E)$ is irreducible and generically smooth.
We briefly discuss about the singularities of the Quot scheme $Q_{k,d}(E)$.
In the case when $C$ is a smooth projective curve, 
it is proved in \cite[Lemma 6.1]{gs24} that the Quot scheme $Q_{k,d}(E)$ is local complete intersection when $d \gg 0$.
The same argument does not extend to the case of a nodal curve.
However, we can show that an open subset of $Q_{k,d}(E)$ is local complete intersection when $d \gg 0$.
Let $Q^0_{k,d}$ denote the following open locus in the Quot scheme $Q_{k,d}(E)$,
$$ Q^0_{k,d}:= \{ [q: E \to F] \in Q_{k,d}(E) : \text{ the stalk of $F$ at each node is free} \}\,.$$
\begin{lemma}\label{open locus lci}
    For $d \geqslant d_Q(E,k)$, the locus $Q^0_{k,d}$ is a local complete intersection and integral.
\end{lemma}
\begin{proof}
    As $d \geqslant d_Q(E,k)$, the open locus $Q^0_{k,d}(E)$ is irreducible.
    So $\dim_q Q^0_{k,d}(E)$ is independent of the closed point $q \in Q^0_{k,d}(E)$. 
    Let $q =[E \to F] \in Q^0_{k,d}(E)$ be a closed point
    and $S_F$ denote the kernel of $q$.
    Then the sheaf $S_F$ is locally free.
    Consequently $\Ext^2(S_F,F) =0$.
    Using the dimension bound \eqref{dimension bound}, we have the following inequalities,
    \begin{equation}\label{equation for lci}
        h^0(S_F^\vee \otimes F) \geqslant \dim_q Q^0_{k,d}(E) \geqslant \chi(S_F^\vee \otimes F)\,.
    \end{equation}
    From the flatness of the Universal sheaves, it follows that 
    the Euler characteristic $\chi(S_F^\vee \otimes F)$ is independent of the point $q$. 
    Moreover, on the open locus of smooth points of $Q^0_{k,d}(E)$, 
    we have $h^1(S_F^\vee \otimes \F) =0$ and 
    hence equality holds in both the places in \eqref{equation for lci}.
    As $\dim_q Q^0_{k,d}(E)$ and $\chi(S_F^\vee \otimes F)$ both are independent of the point $q$,
    we have 
    $$\dim_q Q^0_{k,d}(E) = \chi(S_F^\vee \otimes F) \quad \textrm{ for any } q = [E \to F]\in Q^0_{k,d}(E)\,.$$
    Using Proposition \cite[proposition 2.2.8]{HL},
    we conclude that $Q^0_{k,d}(E)$ is a local complete intersection.
For integrality, we note that the locus $Q^0_{k,d}(E)$ satisfies Serre's conditions $S_1$ and $R_0$.
\end{proof}
However, there exist points $q=[E\to F]$ outside the open locus 
$Q^0_{k,d}(E)$ such that $\Ext^2(S_F,F) \neq 0$.
For such points, the inequality \eqref{equation for lci} is not true. So the argument of Lemma \ref{open locus lci} fails outside the locus $Q^0_{k,d}(E)$.

\newcommand{\etalchar}[1]{$^{#1}$}

\end{document}